\documentclass[a4paper]{article}
\usepackage[utf8]{inputenc}
\usepackage{amssymb, amsmath, tikz, tikz-cd, amsthm, stmaryrd, cite}
\usepackage{graphicx, appendix}
\usepackage[new]{old-arrows}
\usepackage{setspace}

\theoremstyle{plain}
\newtheorem{thm}{Theorem}[section]
\newtheorem{prop}[thm]{Proposition}
\newtheorem{lem}[thm]{Lemma}

\theoremstyle{definition}
\newtheorem{dfn}[thm]{Definition}
\newtheorem{ex}[thm]{Example}
\newtheorem{rem}[thm]{Remark}
\numberwithin{equation}{section}
\def\al{\alpha}
\def\be{\beta}
\def\ga{\gamma}

\def\ep{\epsilon}

\def\si{\sigma}

\def\Om{\Omega}
\def\Ga{\Gamma}

\def\Th{\Theta}

\def\bu{\bullet}

\def\R{{\mathbin{\mathbb R}}}
\def\Z{{\mathbin{\mathbb Z}}}
\def\Q{{\mathbin{\mathbb Q}}}

\def\C{{\mathbin{\mathbb C}}}

\def\uE{{\mathbb{E}^\bu}}
\def\ch{{\rm ch}}
\def\Map{{\rm Map}}
\def\Gr{{\rm Gr}}

\def\A{{\mathbin{\cal A}}}
\def\M{{\mathbin{\cal M}}}

\def\E{{\mathbin{\cal E}}}
\def\X{{\mathbin{\cal X}}}

\def\ra{\rightarrow}

\def\cExt{\mathcal{E}{\rm xt}^\bu}

\SetSymbolFont{stmry}{bold}{U}{stmry}{m}{n}
\DeclareFontFamily{U}{rsfs}{} 
\DeclareFontShape{U}{rsfs}{n}{it}{<->
rsfs10}{} \DeclareSymbolFont{mscr}{U}{rsfs}{n}{it}
\DeclareSymbolFontAlphabet{\scr}{mscr}
\def\mathscr{\scr}

\title{The homology of moduli stacks of complexes}
\author{Jacob Gross}
\date{}
\begin{document}

\maketitle

\begin{abstract}
We compute the rational homology of the moduli stack $\M$ of objects in the derived category of certain smooth complex projective varieties $X$ including toric varieties, flag varieties, curves, surfaces, and some 3- and 4-folds. We identify Joyce's vertex algebra construction on 
$H_\ast(\M,\Q)$ with a generalized super-lattice vertex algebra associated to $K^0_{\rm top}(X^{\rm an}) \oplus K^1_{\rm top}(X^{\rm an})$. 
\end{abstract}

\tableofcontents

\section{Introduction}

In \cite{Joy2} Joyce constructs a graded vertex algebra structure on the homology of moduli stacks $\M$ of objects in certain dg-categories $\A$. This is done by writing down explicit fields that depend on the following extra data: 
    \begin{itemize}
        \item a quotient $K(\A)$ of the Grothendieck group $K^0(\A)$,
        \item signs $\epsilon_{\alpha,\beta} \in \{\pm 1\}$ for all $\alpha,\beta \in K(\mathcal{A})$, and
        \item a perfect complex $\Theta^\bullet \in \text{Perf}(\mathcal{M} \times \mathcal{M})$.
    \end{itemize}
These data are to satisfy certain conditions \cite[4.1-2]{Joy2}, in which case we say that $\M$ is {\it a singular ring}. This terminology comes from Borcherds' definition of vertex algebras as commutative rings with respect to the singular tensor product in a certain category \cite{Bor2}. 

One goal of this theory is to provide a framework for expressing wall-crossing formulas for Donaldson--Thomas type invariants of Calabi--Yau 4-folds \cite{BoJo} \cite{CaLe}, Donaldson--Thomas invariants of Fano 3-folds \cite{Th}, and Donaldson invariants of algebraic surfaces \cite{Mo}. Note that these invariants are non-motivic.

It is already known that the wall-crossing formula for generalized Donaldson--Thomas invariants of Calabi--Yau $3$-folds can be expressed solely in terms of the Lie bracket of a Ringel--Hall algebra of stack functions \cite[Thm.~3.14]{JoSo}. However, the formalism of Ringel--Hall algebras of stack functions is only useful for motivic invariants. Nonetheless, there are a series of conjectures to the effect that the formula \cite[p.~31]{Joy1} is, in a sense, a universal wall-crossing formula (see Gross--Joyce--Tanaka \cite{GJT}). To write these formulas, one needs a suitable Lie bracket. Vertex algebras are one source of Lie algebras. 

Given a vertex algebra, for each $n \in \Z$, there is a binary operation $(v,w) \mapsto v_n(w)$. One could regard these operations as something like Lie brackets. From that perspective, a vertex algebra is an infinitary Lie algebra possessing infinitely many Lie brackets that satisfy infinitely many identities which resemble anti-symmetry and the Jacobi identity. In fact, after modding out a certain sub-module $DV$, $(v,w) \mapsto v_0(w)$ is a genuine Lie bracket on $V/DV$. 

The present document is dedicated to computing examples. The example when $\M$ is the moduli stack of objects in the derived category of representations of a finite quiver without vertex loops $Q=(Q_0,Q_1,t,h)$ was worked out by Joyce \cite{Joy2}. In this case, $H_\ast(\M,\Q)$ is the graded lattice vertex algebra associated to the lattice $\Z^{Q_0}$ whose integral structure is given by the Euler form when $\A$ is $2n$-Calabi--Yau and otherwise given by the symmetrized Euler form. We consider the moduli stack $\M$ of objects in the $\C$-dg-category ${\rm Perf}(X)$ of perfect complexes of coherent sheaves on a smooth complex projective variety $X$. 

We first combine several known results to conclude that the Betti realization of $\M$ has the homotopy type of the semi-topological K-theory space of $X$. If the identity component of $\Om^\infty K^{\rm sst}(X)$ has finite Betti numbers one can apply the Milnor--Moore theorem to compute the rational homology of $\M$--obtaining 
    \begin{equation}
    	H_\ast(\M,\Q) \cong \Q[K^0_{\rm sst}(X)] \otimes {\rm SSym}_\Q[\bigoplus_{i > 0} K^i_{\rm sst}(X)],
    \end{equation}
where $\Q[-]$ denotes the group $\Q$-algebra and ${\rm SSym}_\Q[-]$ denotes the free super-symmetric (i.e. commutative-graded) $\Q$-algebra. 

Unfortunately, computing semi-topological K-groups is difficult. There is a certain class of smooth complex projective varieties, however, for which $K^i_{\rm sst}(X)$ is isomorphic to $K^i_{\rm top}(X^{\rm an})$ for all $i > 0$. We call this class D (Def. \ref{hom5dfn5}). Projective varieties in class D include curves, surfaces, toric varieties, flag varieties, and rational 3- and 4-folds. Varieties with non-trivial Griffiths groups, such as general Calabi--Yau 3-folds, cannot be in class D. For varieties that are in class D, there is a homotopy equivalence between the Betti realization of the moduli space of objects in their derived categories and the complex topological K-theory space of their underlying analytic spaces. Using Brown-Szczarba and Haefliger's models of the rational homotopy types of mapping spaces and evaluation maps, we find that the rational cohomology of moduli stacks of perfect complexes of coherent sheaves on varieties in class D is freely generated by K\"unneth components of Chern classes of the universal complex (Thm. \ref{hom5thm13}). In these cases, we can compute an explicit graded vertex algebra. 

\begin{thm}[see Thm. \ref{hom6thm8}]
\label{hom1thm1}
Let $X$ be a smooth complex projective variety in class D and let $\M$ denote the moduli stack of objects in Perf$(X)$. Let $\uE$ denote the universal complex over $X \times \M$, let $\cExt = \R \pi_\ast(\pi^\ast(\uE)^\vee \otimes^{\mathbb{L}} \pi^\ast(\uE))$, and let $\chi : K_{\rm top}^0(X^{\rm an}) \oplus K_{\rm top}^1(X^{\rm an}) \times K_{\rm top}^0(X^{\rm an}) \oplus K_{\rm top}^1(X^{\rm an}) \ra \Z$ denote the Euler form
    $$\chi(v,w) = \int_{X^{\rm an}} \ch(v)^\vee \cdot \ch(w) \cdot {\rm Td}(X^{\rm an}).$$
 Then
\begin{itemize}
    \item[1.] The shifted homology $\hat{H}_\ast(\M,\Q)$ is made into a graded vertex $R$-algebra by taking $K({\rm Perf}(X)) = K^0_{\rm sst}(X)$, taking $\Th^\bu = (\cExt)^\vee \oplus \si^\ast (\cExt)[2n]$ for some $n \in \Z$, taking symmetric form to be the restriction of the symmetrisation $\chi_{\rm sym}(v,w) = \chi(v,w) + \chi(w,v)$ of $\chi$ to $K^0_{\rm sst}(X)$, and $\ep_{\al,\be} = (-1)^{\chi(\al,\be)}$ in (\ref{hom4eqn1}).  Then $\hat{H}_\ast(\M,\Q)$ is isomorphic, as a graded vertex algebra, to
        \begin{equation}
        \label{hom1eqn2}
            \Q[K^0_{\rm sst}(X)] \otimes {\rm Sym}(K^0_{\rm top}(X^{\rm an}) \otimes t^{-1}\Q[t^{-1}]) \otimes \bigwedge (K^1_{\rm top}(X^{\rm an}) \otimes t^{-\frac{1}{2}}\Q[t^{-1}]),
        \end{equation}
    where (\ref{hom1eqn2}) is given the structure of the generalized super-lattice vertex $R$-algebra associated to  $(K_{\rm top}^0(X^{\rm an}) \oplus K_{\rm top}^1(X^{\rm an}), \chi_{\rm sym})$ and the inclusion $K^0_{\rm sst}(X) \hookrightarrow K^0_{\rm top}(X^{\rm an})$.
    \item[2.] Suppose that $X$ is 2n-Calabi--Yau and that for all $\al,\be \in K^0_{\rm sst}(X)$ we are given signs $\ep_{\al,\be} \in \{\pm 1\}$ such that $\{\ep_{\al,\be}\}$ is a solution of (\ref{hom3eqn3})-(\ref{hom3eqn5}). Then shifted\footnote{Note that the shift in 2. is different than the shift in 1.} homology $\hat{H}_\ast(\M,\Q)$ is made into a graded vertex $R$-algebra by taking $K({\rm Perf}(X)) = K^0_{\rm sst}(X)$, taking $\Th^\bu = (\cExt)^\vee$, taking symmetric form to be $\chi$, and taking $\ep_{\al,\be}$ as signs in (4.3). Then $\hat{H}_\ast(\M,\Q)$ is isomorphic, as a graded vertex $R$-algebra, to
        \begin{equation}
        \label{hom1eqn3}
            \Q[K^0_{\rm sst}(X)] \otimes {\rm Sym}(K^0_{\rm top}(X^{\rm an}) \otimes t^{-1}\Q[t^{-1}]) \otimes \bigwedge (K^1_{\rm top}(X^{\rm an}) \otimes t^{-\frac{1}{2}}\Q[t^{-1}]),
        \end{equation}
    where (\ref{hom1eqn3}) is given the structure of the generalized super-lattice vertex $R$-algebra associated to the $(K_{\rm top}^0(X^{\rm an}) \oplus K_{\rm top}^1(X^{\rm an}), \chi)$ and the inclusion $K^0_{\rm sst}(X) \hookrightarrow K^0_{\rm top}(X^{\rm an})$. Up to isomorphism, this graded vertex $R$-algebra is independent of the representative of the cohomology class $[\ep] \in H^2(K^0_{\rm sst}(X),\Z/2\Z)$ that $\{\ep_{\al,\be}\}_{\al,\be \in K^0_{\rm sst}(X)}$ defines.
\end{itemize}
\end{thm}

\begin{ex}
\label{hom1ex2}
Let $X$ be a K3 surface and $\M,$ as in Theorem \ref{hom1thm1}. Then $K^0_{\rm top}(X^{\rm an})$ is a lattice and $K^1_{\rm top}(X^{\rm an}) \cong 0$ so that the fermionic part of (\ref{hom1eqn3}) vanishes. Thus $\hat{H}_\ast(\M,\Q)$ is the ordinary (graded) lattice vertex $R$-algebra associated to the Mukai lattice (with restricted group algebra $\Q[K^0_{\rm sst}(X)] \subset \Q[K^0_{\rm top}(X^{\rm an})]$).
\end{ex}

\begin{rem}
There was been much recent work on enumerative invariant theories taking values in K-theory and cobordism (see Arbsefeld \cite{Arb}, Cao--Kool--Monavari \cite{CKM}, G\"ottsche--Kool \cite{GoKo}, G\"ottsche--Nakajima--Yoshioka \cite{GNY}, Laarakker \cite{Laa}, Okounkov \cite{Ok}, Shen \cite{Shen}, and Thomas \cite{Th2}). Theorem \ref{hom6thm8} carries over wholesale for a rational complex-oriented cohomology theory $E$. This is because all rational complex-oriented spectra are canonically equivalent to graded Eilenberg--Maclane spectra. Without rationalizing, one should obtain a generalized vertex algebra called a {\it vertex $F$-algebra} where $F$ is the formal group law of $E$ (see forthcoming work of Gross--Upmeier \cite{GrUp} and Meinhardt \cite{Mein}).
\end{rem}

In section 2, we review the necessary background from algebraic topology: homotopy-theoretic group completions and some rational homotopy theory. Group completion theory allows us to reduce questions about maps from $\M^{\rm Betti}$ to a group-like H-space (such as a classifying space) to questions about maps from the underlying analytic space of the $\C$-ind-scheme Map$_{\rm alg}(X, \Gr)$. This $\C$-ind-scheme parameterizes globally generated algebraic vector bundles on $X$. This can turn certain questions about complexes into questions about vector bundles. For example, this technique was used in \cite{CaGrJo} to prove orientability of moduli spaces of coherent sheaves on Calabi--Yau 4-folds. This kind of argument is also used in section 5.

In section 3 we introduce graded vertex algebras. We give three examples: a graded vertex algebra associated to a finitely generated abelian group equipped with a symmetric $\Z$-bilinear form and a map from another finitely generated abelian group, a graded vertex algebra associated to a finitely generated abelian group which is equipped with an anti-symmetric $\Z$-bilinear form, and tensor products of graded vertex algebras. A {\it generalized super-lattice vertex algebra}, as referenced in Theorem \ref{hom1thm1}, is a tensor product graded vertex algebra which is associated to a generalized super-lattice that is equipped with a super-symmetric $\Z$-bilinear form. Usually one associates vertex algebras to lattices rather than finitely generated abelian groups. However we will want to associate graded vertex algebras to K-theory groups, which can have torsion. This generalization is easy.

In section 4 we define the Betti homology of a higher $\C$-stack using Simpson's Betti realization of simplicial presheaves and then briefly outline Joyce's graded vertex algebra construction. Then, we compute the rational homology of the moduli stack $\M$ of perfect complexes of coherent sheaves on a smooth complex projective variety $X$ in class D. In section 5, we use the description of $\hat{H}_\ast(\M, \Q)$ in terms of Chern characters to give explicit formulas for the state-to-field correspondences. From this, we are able prove Theorem \ref{hom1thm1} by calculation. \\

\noindent \emph{Acknowledgments.} The author thanks Dominic Joyce for supervising this project as well as for providing numerous suggestions and revisions. We thank Bertrand To\"en for suggesting we apply Blanc's work.  We also thank Yalong Cao, Eric Friedlander, Ian Grojnowski, Mikhail Kapranov, Martijn Kool, Kobi Kremnitzer, Sven Meinhardt, Jan Steinebrunner, Richard Thomas, Yukinobu Toda, and Markus Upmeier for helpful discussions. The author acknowledges support from the Simons Collaboration on Special Holonomy in Geometry, Analysis, and Physics during the time this document was produced.

\section{Algebraic topology background}

In this section, we give the necessary background on homotopy-theoretic group completions and rational homotopy theory.

\subsection*{Homotopy-theoretic group completion}

In this subsection, we review parts of the theories of H-spaces as in \cite{Stash}, homotopy-theoretic group completions as in \cite{May2}, and semi-topological K-theory as in \cite{FrWa1}.

An {\it H-space} is a CW complex $X$ together with a {\it multiplication} $m : X \times X \ra X$, and a {\it unit} $e : {\rm pt} \ra X$ such that $m$ is homotopy associative $m \circ (m \times 1_X) \simeq m \circ (1_X \times m)$ and $e$ is homotopy unital $1_X \simeq m \circ (1_X \times e) \simeq m \circ (e \times 1_X)$.  All our H-spaces are assumed to be homotopy commutative. An {\it H-map} $(X,m_X,e_X) \ra (Y,m_Y,e_Y)$ is a continuous map $f : X \ra Y$ such that $(f \times f) \circ m_y \simeq f \circ m_X$ and $f(e_X) = e_Y$.

If $X$ is an H-space then $\pi_0(X)$ is an abelian monoid. One says that an H-space $X$ is {\it group-like H-space} if $\pi_0(X)$ is an abelian group. This is equivalent to the condition that $m : X \times X \ra X$ admits a homotopy inverse. Any connected H-space is group-like \cite[Lem~9.2.2]{MaPo}. The following theorem is called the Milnor--Moore theorem and it is a powerful tool for computing the rational homology of connected finite type\footnote{A topological space is said to be {\it finite type} if all of its Betti numbers are finite.} H-spaces.

\begin{thm}[{see \cite[Thm.~9.2.5]{MaPo}, \cite{MiMo}}]
\label{hom2thm8}
Let $X$ be a finite type connected H-space. Then there is a natural isomorphism
    $$H_\ast(X,\Q) \cong {\rm SSym}[\pi_\ast(X) \otimes \Q]$$
of commutative-graded $\Q$-Hopf algebras.    
\end{thm}

\noindent Given an H-space that is not group-like, one can sometimes complete it to a group-like H-space.

\begin{dfn}
\label{hom2dfn9}
Let $X$ be an H-space. Then a {\it homotopy-theoretic group completion of $X$} is an H-map $\Xi : X \ra X^+$ to a group-like H-space $X^+$ such that 
\begin{itemize}
    \item $\pi_0(\Xi) : \pi_0(X) \ra \pi_0(X^+)$ is group completion of the monoid $\pi_0(X)$, and
    \item $H_\ast(\Xi,A) : H_\ast(X,A) \ra H_\ast(X^+,A)$ is module localisation by the natural $\pi_0(X)$ action, for every abelian group $A$.
\end{itemize}
\end{dfn}

\noindent Note that homotopy-theoretic group completions are not defined by a universal property. Caruso--Cohen--May--Taylor have demonstrated that homotopy-theoretic group completions have a {\it weak} universal property \cite{CCMT}. Two continuous maps $f,g : X \ra Y$ are said to be {\it weakly homotopy equivalent}, written $f \simeq^w g$ if for every continuous map $h : Z \ra X$ from a finite CW complex $Z$ there is a homotopy $f \circ h \simeq g \circ h$. 

\begin{prop}[see {\cite[Prop.~1.2]{CCMT}}] 
\label{hom2prop10}
Let $\Xi : X \ra X^+$ be a homotopy-theoretic group completion such that $\pi_0(X)$ contains a countable cofinal sequence. Then for any group-like H-space $Z$ and any weak H-map $f : X \ra Z$ there exists an H-map $g : X^+ \ra Z$, unique up to weak homotopy, such that $g \circ \Xi \simeq^w f$.
\end{prop}

\noindent Suitably enhanced H-spaces admit homotopy-theoretic group completions. Two essentially equivalent such enhancements are Segal's $\Ga$-spaces \cite{Seg} and May's $E_\infty$-spaces \cite{May1}. This formalism is useful in defining the semi-topological K-theory of varieties. Let $\Gr$ denote the complex ind-scheme $$\Gr = \coprod_{n \geq 0} \Gr_n(\C^\infty).$$ Taking the underlying complex analytic space of  $\Gr$ gives
    $$\Gr^{\rm an} \simeq \coprod_{n \geq 0} BU(n),$$
which is an $E_\infty$-space whose homotopy-theoretic completion is $\coprod_{n \geq 0} BU(n) \ra BU \times \Z$. What is more, given any smooth complex projective variety $X$ the underlying analytic space of the mapping ind-scheme Map$(X,\Gr)^{\rm an}$ is an $E_\infty$-space (see \cite[Prop.~2.2]{FrWa2}). The homotopy-theoretic group completion of Map$(X,\Gr)^{\rm an}$ is called the {\rm semi-topological K-theory space of $X$} and is written $\Om^\infty K^{\rm sst}(X)$. The {\it $n^{\rm th}$ semi-topological K-group of $X$} is defined by $K^n_{\rm sst}(X) := \pi_n(\Om^\infty K^{\rm sst}(X))$.

Complex topological K-theory can be described similarly. Let $M$ be a CW complex. Then the mapping space Map$_{C^0}(M, \coprod_{n \geq 0} BU(n))$ is an $E_\infty$-space whose homotopy-theoretic group completion is equivalent to $\Om^\infty K^{\rm top}(M) := {\rm Map}_{C^0}(M, BU \times \Z)$. For a smooth complex projective variety $X$ the natural map Map$_{\rm alg}(X,\Gr) \ra {\rm Map}_{C^0}(X^{\rm, an}, \coprod_{n \geq 0} BU(n))$ induces a K-theory comparison morphism $\Om^\infty K^{\rm sst}(X) \ra \Om^\infty K^{\rm top}(X^{\rm an})$.

\subsection*{Rational homotopy theory}
The subject of rational homotopy theory was initiated by the foundational works of Quillen \cite{Qui2} and Sullivan \cite{Sul} on the algebraicisation of the homotopy types of rational spaces. The standard encyclopaedic reference for this subject in \cite{FeHaTh}. We recall only a small portion of rational homotopy theory that will be used in later proofs. Our only use of the theory will be to  compute the rational cohomology of an evaluation map. The reader may wish to skip this subsection.

If $X$ is a real manifold, there is an $\mathbb{R}$-cdga $\mathbb{R}$-cdga $(\Om^\ast_{\rm dR}(X),d_\text{dR})$ such that $H^\ast(\Om^\ast_{\rm dR})$. If $X$ is any topological space there is a $\mathbb{Q}$-cdga $(A_\text{PL}(X),d)$ called the {\it algebra of rational polynomial forms} such that $H^\ast((A_\text{PL}(X),d)) \cong H^\ast(X, \Q)$ (see \cite[\S~10]{FeHaTh}). When $X$ has the homotopy type of a real manifold there is an isomorphism $A_\text{PL}(X) \otimes \mathbb{R} \cong \Om^\ast_\text{dR}(X)$.

A {\it Sullivan algebra} is a $\Q$-cdga of the form $({\rm SSym}_\Q[V], d)$ that satisfies a nilpotency condition (see \cite[p.~138]{FeHaTh}). A {\it Sullivan model} of a $\Q$-cdga $(A,d_A)$ is a quasi-isomorphism $({\rm SSym}_\Q[V], d_V) \ra (A, d_A)$ from a Sullivan algebra. A {\it Sullivan model} of a $\Q$-cdga morphism $(A, d_A) \ra (B, d_B)$ is a $\Q$-cdga morphism $({\rm SSym}_\Q[V], d_V) \ra ({\rm SSym}_\Q[W], d_W)$ making the diagram
\begin{center}
	\begin{tikzcd}
		({\rm SSym}_\Q[V], d_V) \arrow{r} \arrow{d} & ({\rm SSym}_\Q[W], d_W) \arrow{d} \\
		(A, d_A) \arrow{r} & (B, d_B)
	\end{tikzcd}
\end{center}
commute, where $({\rm SSym}_\Q[V], d_V) \ra (A, d_A)$ and $({\rm SSym}_\Q[W], d_W) \ra (B, d_B)$ are both Sullivan models. Every Sullivan algebra is isomorphic to a particularly nice Sullivan algebra called its {\it minimal model}  \cite[\S12, Def.~1.3]{FeHaTh} which is unique up to isomorphism. A {\it Sullivan model} of a topological space $X$ is a Sullivan model of $(A_{\rm PL}(X), d)$ and a {\it minimal model} of $X$ is a minimal model of $(A_{\rm PL}(X), d)$. A {\it model} of a topological space $X$ is a $\Q$-cdga which admits a Sullivan model quasi-isomorphic to a Sullivan model of $X$. Models of continuous maps are defined analogously. Models of H-spaces very well-understood.

\begin{prop}
\label{hom2prop11}
Let $X$ be a connected and simply-connected finite type topological space. Then the minimal model of $X$ is of the form $({\rm SSym}_\Q[\pi_\ast(X)^\vee], d)$. If $X$ is moreover an H-space then its minimal model has vanishing differential.
\end{prop}

The rational homotopy theory of function spaces and evaluation maps is also now quite well-understood, having been studied by authors such as Buijs--Murillo \cite{BuMu}, F\'elix \cite{Fe}, F\'elix--Tanr\'e \cite{FeTa}, Kotani \cite{Kot}, Kuribayshi \cite{Kur}, Smith \cite{Smith}, and Vigu\'e-Poirrier \cite{ViPo}.

\begin{prop}[{see Buijs--Murillo \cite[Prop.~4.4, Thm.~4.5]{BuMu}, Brown--Szczarba \cite[Thms.~1.5, 6.1]{BrSz}, Haefliger \cite[Thm.~3.2]{Hae}}]
\label{hom2prop13}
Let $X$ be a finite CW complex and let $Y$ be a connected and simply-connected finite type CW complex. Let $({\rm SSym}_\Q[\pi^\ast(Y)], d_Y)$ be the minimal model for $Y$. Let $(A, d_A)$ be a finite-dimensional model for $X$. Let $A_\ast$ be the graded differential $\Q$-coalgebra whose $i^{\rm th}$ component is ${\rm Hom}(A_{-i}, \Q)$. There is a canonical $\Q$-cdga morphism
    $$\ep' : {\rm SSym}_\Q[\pi^\ast(Y)] \longrightarrow A \otimes {\rm SSym}_\Q[A_\ast \otimes \pi^\ast(Y)]$$
given in terms of an additive basis $v_1,\dots,v_n$ of $A$ and dual basis $v_1^\vee,\dots,v^\vee_n$ of $A_\ast$ by
    $$\ga \mapsto \sum_{i} v_i \otimes (v_i^\vee \otimes \ga).$$
There is a unique differential $d$ on ${\rm SSym}_\Q[A_\ast \otimes \pi^\ast(Y)]$ such that $\ep'$ is a $\Q$-cdga morphism. Note that $A_\ast \otimes \pi^\ast(Y)$ is not necessarily connective. Let $I$ denote the ideal of $A_\ast \otimes \pi^\ast(Y)$ generated by elements of negative degree and their differentials. Let $W$ be the quotient of $A_\ast \otimes \pi^\ast(Y)$ by $I$. Then
\begin{itemize}
    \item the $\Q$-cdga $({\rm SSym}_\Q[W],d)$ is a model for the connected component of the space of continuous maps $X \ra Y$ containing the constant map, and 
    \item $\ep'$ induces a $\Q$-cdga morphism $\ep : {\rm SSym}_\Q[\pi^\ast(Y)] \ra A \otimes {\rm SSym}_\Q[W]$ which is a model for the evaluation map.
\end{itemize}

\end{prop}

\noindent If $X$ is a CW complex such that there exists a model (not necessarily minimal) of $X$ with vanishing differential, then $X$ is said to be {\it formal}. H-spaces, symmetric spaces, and compact K\"ahler manifolds are examples of formal spaces.

\section{Graded vertex algebras}

Some references for the basic vertex algebra theory presented here are \cite{Bor1}, \cite{Kac} and \cite{LeLi}. We, after \cite{Joy2}, work with $\Z$-graded vertex $R$-algebras. Throughout, $R$ will denote a fixed commutative unital ring.

A formal power series $\sum_{n \in \Z} v_n z^{-n-1}$ is called a {\it field} if $v_n = 0$ for $n \gg 0$.  For a graded $R$-module $V$ we write $\mathcal{F}(V)$ for the graded $R$-module of End$_R(V)$-valued fields. Grade $\mathcal{F}(V)$ so that a field of degree $n$ is a map $V \ra V[[z,z^{-1}]]$ of degree $n$, where $z$ has degree $-2$.

\begin{dfn} 
\label{hom3dfn1}
A \emph{graded vertex $R$-algebra} is the data of
\begin{itemize}
	\item a $\Z$-graded $R$-module $V = \bigoplus_{n \in \Z} V_n$ called a {\it space of states},
        \item a distinguished vector $| 0 \rangle \in V_0$ called the {\it vacuum vector}, 
        \item a graded $R$-linear map $V \ra \mathcal{F}(V)$ called a {\it state-to-field correspondence}, and 
        \item a graded operator $e^{zD} : V \ra V[[z]]$
\end{itemize} 
such that, for all $u,v,w \in V$,
\begin{itemize}
	\item $Y(| 0 \rangle ,z) v = v$, $Y(u,z) | 0 \rangle = e^{zD} v$,
	\item there exists $N \gg 0$ such that
		$$(z_1 + z_2)^N Y(Y(u,z_1)v,z_2) w = (z_1 + z_2)^N Y(u,z_1 + z_2) \circ Y(v,z_2) w,$$
		and
	\item $Y(u,z)v = (-1)^{{\rm deg}(u) {\rm deg}(v)} e^{zD} \circ Y(v,-z)u.$
\end{itemize}
\end{dfn}

\begin{dfn}
\label{hom3dfn2}
Let $(V,Y_V, |0\rangle_V)$ and $(W,Y_W, |0\rangle_W)$ be a graded vertex $R$-algebra. Their {\it tensor product} is a graded vertex algebra $(V\otimes W,Y_{V \otimes W}, |0\rangle_{V \otimes W})$ whose space of states is the graded tensor product $V \otimes W$, whose vacuum vector is given by $|0\rangle_{V \otimes W} = | 0 \rangle_V \otimes | 0 \rangle_W$, and whose state-to-field correspondence is given by $$Y_{V \otimes W}(v \otimes w, z)(u \otimes t) = (-1)^{({\rm deg}(v) + {\rm deg}(w)) {\rm deg}(u)} Y_V(v,z)u \otimes Y_W(w,z)t.$$
\end{dfn}

It is often more useful to describe a vertex algebra in terms of a generating set of fields, rather than by describing all fields. 

\begin{dfn}
\label{hom3dfn3}
Let $(V,| 0 \rangle)$ be a pair consisting of a graded $R$-module $V$ and distinguished element $| 0 \rangle \in V_0$. Let $\{a^i(z) = \sum_{n \in \Z} a^i_n z^{-n-1}\}_{i \in I}$ be a set of ${\rm End}_R(V)$-valued fields. Let $a^i = a^i(z) | 0 \rangle |_{z=0}$. Then $\{a^i(z)\}_{i \in I}$ is said to {\it generate} $(V,| 0 \rangle)$ if the collection of vectors of the form $a^{i_1}_{n_1} \dots a^{i_s}_{n_s} | 0 \rangle$ span $V$.
\end{dfn}

\begin{prop}[{Reconstruction Theorem \cite[Thm.~4.5]{Kac}}]
\label{hom3prop4}
Let $V$ be a graded $R$-module with a distinguished vector $| 0 \rangle \in V$. Let $\{a^i(z)\}_{i \in I}$ be a mutually local collection of ${\rm End}_R(V)$-valued fields that generate $(V, | 0 \rangle)$. Let $T : V \ra V$ be a derivation of degree $2$ such that $T | 0 \rangle = 0$ and such that the collection is {\it $T$-covariant} $[T, a^i(z)] = \partial_z a^i(z)$. Then there is a unique graded vertex $R$-algebra structure on $V$ such that $| 0 \rangle$ is the vacuum vector and such that the state-to-field correspondence maps $a^i(z)| 0 \rangle|_{z=0} \mapsto a^i(z)$.
\end{prop}

A {\it generalized integral lattice} is a finitely generated abelian group equipped with a symmetric $\Z$-bilinear form. Let $(A^+,\chi^+)$ be a generalized integral lattice and let $\iota : B^+ \ra A^+$ be a map of finitely generated abelian groups. Let $\mathfrak{h} := A^+ \otimes_\Z R$. Then $\chi^+$ extends to an $R$-valued symmetric bilinear form on $\mathfrak{h}$ which, by abuse of notation, we will also call $\chi^+$. Write $\hat{\mathfrak{h}} = \mathfrak{h} \otimes R[t,t^{-1}] \oplus R \underline{K}$ for the Lie algebra with bracket
    $$[h \otimes t^n, h' \otimes t^m] = m \delta_{m,-n} \chi^+(h,h') \underline{K}, \hspace{1em} [h \otimes t^m,\underline{K}] = 0$$
for all $h,h' \in \mathfrak{h}, n,m \in \Z$.
There is a representation $\rho_1 : \hat{\mathfrak{h}} \ra {\rm Sym}_R(A^+ \otimes t^{-1}R[t^{-1}])$ such that $\mathfrak{h} \otimes t^{-1}R[t^{-1}]$ acts by multiplication, $\mathfrak{h}[t] \cdot 1 = 0$, and $\underline{K}$ acts as the identity. There is another Lie algebra representation $\rho_2 : \hat{\mathfrak{h}} \ra R[B^+]$ where $\underline{K}$ acts by zero and $h \otimes t^n$ acts on $e^\al$ as multiplication by $\delta_{n,0} \chi^+(\iota(\al),h)$. Write $V_{A^+,B^+} := R[B^+] \otimes {\rm Sym}_R(A^+ \otimes t^{-1}R[t^{-1}])$. There is a representation $\rho : \hat{\mathfrak{h}} \ra V_{A^+,B^+}$ defined by $\rho = \rho_1 \otimes 1 + 1 \otimes \rho_2$. Given $v \in \mathfrak{h}$ write $$v(z) := \sum_{n \in \Z}  \rho(v \otimes t^n) z^{-n-1}$$
and 
    $$\Ga_\al(z) := e^\al z^{\al_0} {\rm exp}(-\sum_{j < 0} \frac{z^{-j}}{j} \al_j) {\rm exp}(-\sum_{j > 0} \frac{z^{-j}}{j} \al_j) c_\al,$$
where $c_\al$ are operators $c_\al : V_{A^+,B^+} \ra V_{A^+,B^+}$ such that 
\begin{equation}
\label{hom3eqn1}
    c_0 =1, \hspace{1em} c_\al | 0 \rangle = | 0 \rangle, \hspace{1em} [v_n,c_\al]=0 \hspace{1em} (v \in \mathfrak{h}, n \in \Z),
\end{equation}
and given $\al \in B^+$
\begin{equation}
\label{hom3eqn2}
    e^\al c_\al e^\be c_\be = (-1)^{\chi^+(\iota(\al),\iota(\be)) + \chi^+(\iota(\al),\iota(\al)) \chi^+(\iota(\be),\iota(\be))} e^\be c_\be e^\al c_\al .  
\end{equation}
Then the collection of fields $v(z)$ and $\Ga_\al(z)$ generate a graded vertex $R$-algebra structure on $V_{A^+,B^+}$. For any given solution of (\ref{hom3eqn1}) and (\ref{hom3eqn2}) there is a unique such graded vertex $R$-algebra structure (cf. \cite[Thm.~5.4]{Kac}). The only solutions of (\ref{hom3eqn1}) and (\ref{hom3eqn2}) that concern us are scaling operators
    $$c_\al(e^\be \otimes x) = \ep_{\al,\be} e^\be \otimes x, \hspace{1em} (\al,\be \in B^+)$$
with $\ep_{\al,\be} \in \{\pm 1\}$. For such $c_{\al}$, equations (\ref{hom3eqn1}) and (\ref{hom3eqn2}) are satisfied if and only if for all $\al,\be,\ga \in B^+$
\begin{align}
    \ep_{\al,0} & = \ep_{0,\al} = 0 \label{hom3eqn3} \\
    \ep_{\al,\be} & = (-1)^{\chi^+(\iota(\al),\iota(\be)) + \chi^+(\iota(\al),\iota(\al)))\chi^+(\iota(\be),\iota(\be))} \ep_{\be,\al} \label{hom3eqn4} \\
    \ep_{\be,\ga} \ep_{\be + \ga, \al} & = \ep_{\ga, \al+\be} \ep_{\be,\al} \label{hom3eqn5}.
\end{align}

\noindent There always exist solutions to (\ref{hom3eqn3})-(\ref{hom3eqn5}) \cite[Lem.~4.5]{Joy2} (see also \cite[Cor.~5.5]{Kac} \footnote{Note that Kac proves any solution of (\ref{hom3eqn3})-(\ref{hom3eqn5}) will be unique up to equivalence. However, Kac is working with lattices. In general, the ambiguity in choosing a solution of (\ref{hom3eqn3})-(\ref{hom3eqn5}) is controlled by the 2-torsion in $B^+$ \cite[Thm.~2.27]{JTU}.}). Note that (\ref{hom3eqn3}) and (\ref{hom3eqn5}) imply $\ep : B^+ \times B^+ \ra \Z/2\Z$ is a group 2-cocycle.

One makes $V_{A^+,B^+}$ into a {\it graded} vertex $R$-algebra by declaring $e^\al \otimes (v \otimes t^{-n})$ to be of degree $2n + 2 - \chi^+(\iota(\al),\iota(\al))$. This is called the {\it generalized lattice vertex $R$-algebra} associated to $(A^+,\chi^+)$ and $\iota$. When $B^+=A^+$ with $\iota = {\rm Id}_{A^+}$ and $A^+$ torsion-free this is called the (graded) lattice vertex $R$-algebra associated to the integral lattice $(A^+,\chi^+)$ (see \cite[\S~5.4]{Kac}, \cite[\S~6.4-5]{LeLi}). 

Given a finitely generated abelian group $A^-$ equipped with an anti-symmetric integral bilinear form $\chi^-$ one can construct a similar graded vertex $R$-algebra. Consider the Lie algebra $\mathfrak{h}^- := A^- \otimes t^{\frac{1}{2}}R[t,t^{-1}] \oplus R \underline{K}$ with commutation relations
    $$[v \otimes t^{m+\frac{1}{2}}, w \otimes t^{n+\frac{1}{2}}]_+ = m \chi^-(v,w) \delta_{m,-n} \underline{K}, \hspace{1em} [\underline{K}, \mathfrak{h}^-] = 0.$$
Let $\mathcal{A} = U(\mathfrak{h}^-)/(\underline{K}-1)$ and let $\mathcal{A}_{\geq 0}$ denote the ideal of $\mathcal{A}$ generated by elements of the form $(v \otimes t^{m+\frac{1}{2}}) \cdot 1$ with $v \in A^-$ and $m \geq 0$. Then there is a natural Lie algebra representation $\rho^-$  of $\mathfrak{h}^-$ on $\mathcal{A}/\mathcal{A}_{\geq 0} \cong \bigwedge (A^- \otimes t^{-\frac{1}{2}}R[t^{-1}])$. Make $\bigwedge (A^- \otimes t^{-\frac{1}{2}}R[t^{-1}])$ into a graded $R$-algebra by declaring $v \otimes t^{-i-\frac{1}{2}}$ to be of degree $2i+1$. For $v \in A^-$ write $v(z) = \sum_{n \in \Z} \rho^-(v \otimes t^{n+\frac{1}{2}}) z^{-n-1}$. The collection of fields $\{v(z)\}_{v \in A^+}$ generate a graded vertex $R$-algebra structure on $\bigwedge (A^- \otimes t^{-\frac{1}{2}}R[t^{-1}])$. When $A^-$ is torsion-free and $\chi^-$ is non-degenerate, this construction gives Abe's symplectic Fermionic vertex operator algebra \cite{Abe}.\footnote{Abe's symplectic Fermionic vertex algebras have state space of the form $\bigwedge (A^- \otimes t^{-1}R[t^{-1}])$ rather than $\bigwedge (A^- \otimes t^{-\frac{1}{2}}R[t^{-1}])$. We use the factor $t^{-\frac{1}{2}}$ so that all generators are in odd degrees.}

A {\it generalized integral super-lattice} is a finitely generated abelian group $A$, which is written as a direct sum $A= A^+ \oplus A^-$ of two generalized lattices, equipped with a $\Z$-valued bilinear form $\chi$ such that $\chi^+ := \chi|_{A^+}$ is symmetric and $\chi^- := \chi|_{A^-}$ is anti-symmetric. 

\begin{dfn}
\label{hom3dfn5}
Let $(A=A^+\oplus A^-,\chi)$ be a generalized integral super-lattice. Let $\iota : B^+ \ra A^+$ be a map of finitely generated abelian groups. Let $V_{A^+,B^+}$ be the generalized lattice vertex $R$-algebra associated to $(A^+, \chi^+)$ and $\iota$. Let $V_{A^-}$ denote the graded vertex $R$-algebra associated to $(A^-,\chi^-)$. Write $V_{A} = V_{A^+,B^+} \otimes V_{A-}$ for the tensor product graded vertex $R$-algebra. This is called the {\it generalized super-lattice vertex $R$-algebra associated to $(A=A^+\oplus A^-,\chi)$ and $\iota$}. It is defined up to a choice of solution to (\ref{hom3eqn3})-(\ref{hom3eqn5}) .
\end{dfn}

\noindent Let $(A=A^+ \oplus A^-,\chi)$ be an integral super-lattice, $\iota : B^+ \ra A^+$ be a map of finitely generated abelian groups, $Q^+,Q^-$ be additive bases for $A^+ \otimes R, A^- \otimes R$ respectively, and $Q= Q^+ \cup Q^-$.  Then the map
 \begin{equation}
    \label{hom3eqn7}
        e^\al \otimes (v^+ \otimes t^{-i})^{n_{v^+,i}} \otimes (v^- \otimes t^{-j-\frac{1}{2}})^{m_{v^-,j}} \mapsto e^\al \otimes u_{v^+,i}^{n_{v^+,i}} \cdot u_{v^-,j}^{m_{v^-,j}}
    \end{equation}
defines an isomorphism of graded $R$-algebras
        \begin{equation}
        \label{hom3eqn6}
            V_A \cong R[B^+] \otimes {\rm SSym}_R[u_{v,i} : v \in Q, i \geq 1]
        \end{equation}
    where the right hand side of (\ref{hom3eqn6}) is graded by declaring $e^\al \otimes \prod_{v \in Q, i \geq 1} u^{n_{v,i}}_{v,i}$ to be of degree $\sum_{v^+ \in Q^+, i \geq 1} 2i n_{v^+,i} + \sum_{v^- \in Q^-, i \geq 1} (2i+1) n_{v,i} + 2 - \chi^+(\iota(\al),\iota(\al))$.

\noindent For $v^+ \in A^+ \otimes R$ one writes $$v^+(z) := \sum_{n \in \Z} \rho^+(v^+ \otimes t^n) \otimes {\rm Id} z^{-n-1}$$ and for $v^- \in A^- \otimes R$ one writes $$v^-(z) := \sum_{n \in \Z} {\rm Id} \otimes \rho^-(v^- \otimes t^{n+\frac{1}{2}}) z^{-n-1}.$$ 
\noindent The following is a consequence of the reconstruction theorem.

\begin{prop}
\label{hom3prop6}
Let $A,A^+,A^-,\chi^+,B^+, \iota,Q,Q^+$, and $Q^-$ be as above. Then any graded vertex $R$-algebra with space of states $R[B^+] \otimes {\rm SSym}_R[u_{v,i} : v \in Q, i \geq 1]$ such that
    \begin{equation}
    \label{hom3eqn8}
        (u_{0,v,1})_n = v(z)_n
    \end{equation}
for all $v \in Q, n \in \Z$ is isomorphic to the generalized super-lattice vertex $R$-algebra associated to $(A,\chi)$ and $\iota$.
\end{prop}
\begin{proof}
The proof is nearly identical to that of \cite[Thm.~5.4]{Kac}. First, the collection
$\{v(z), \Ga_\al(z)\}_{v \in Q, \al \in B^+}$ is a set of mutually local generating fields. Moreover, the derivation $T = \partial_t$ makes the collection $\{v(z), \Ga_\al(z)\}_{v \in Q, \al \in B^+}$ into a $T$-covariant collection.  Also (\ref{hom3eqn8}) implies the state-to-field correspondence maps $e^\al \otimes 1 \otimes 1 \mapsto \Ga_\al(z)$ (see \cite[p.~102-3]{Kac}). By Proposition \ref{hom3prop4} (\ref{hom3eqn8}), determines a graded vertex $R$-algebra structure on $R[B^+] \otimes {\rm SSym}_R[u_{v,i} : v \in Q, i \geq 1]$ unique up to isomorphism.

To show that the super-lattice vertex algebra $V_A \cong R[B^+] \otimes {\rm SSym}_R[u_{v,i} : v \in Q, i \geq 1]$ satisfies (\ref{hom3eqn8}), suppose that $v \in Q^+$. Then, as $e^0 \otimes (v \otimes t^{-1}) \otimes 1$ is sent to $u_{0,v,1}$ under (\ref{hom3eqn8}), we have
\begin{align*}
    (u_{0,v,1})_n & = (e^0 \otimes (v \otimes t^{-1}) \otimes 1)_n \\
                  & = {\rm Coeff}_{-n-1} Y(e^0 \otimes (v \otimes t^{-1}) \otimes 1,z) \\
                  & = {\rm Coeff}_{-n-1} Y_+(e^0 \otimes (v \otimes t^{-1}),z) \otimes Y_-(1,z) \\
                  & = {\rm Coeff}_{-n-1}(\sum_{n \in \Z} \rho^+(v \otimes t^n) \otimes {\rm Id} z^{-n-1}) \\
                  & = v(z)_n.
\end{align*}
The $v \in Q^-$ case is similar.
\end{proof}

\section{Homology calculations}

In this section we define the Betti homology of higher $\C$-stacks and review Joyce's vertex algebra construction. We then explicitly compute the rational homology of the moduli stack of objects in the derived category of a smooth complex projective variety $X$.

Note that moduli stacks of perfect complexes of coherent sheaves on varieties are not representable by schemes or stacks, but by higher stacks. For the general theory of higher stacks the reader is referred to \cite{Lur1}, \cite{Prid}, \cite{Sim2}, \cite{To}, or \cite{ToVe}. The construction of moduli stacks of objects in saturated $\C$-dg-categories is due to To\"en--Vaqui\'e \cite{ToVa}.\footnote{We will tacitly be assuming that all stacks are locally of finite type. Importantly, the moduli stack of objects in the $\C$-dg-category Perf$(X)$ is locally of finite type \cite[Thm.~0.2]{ToVa}.}We will first need to say exactly what is meant by the homology of a higher $\C$-stack. 

In \cite{Joy2} there is a list of axioms that a functor $H_\ast(-) : Ho({\rm HArt}) \to R\text{-mod}$ must satisfy in order to endow $H_\ast(\mathcal{M}_\mathcal{A})$ with the structure of a graded vertex algebra over $R$ (provided certain assumptions $\mathcal{A}$ hold). These axioms include the existence of Chern classes, $\mathbb{A}^1$-homotopy invariance, and pushforwards along stack maps that are not necessarily proper. In particular, the algebraic K-homology or the Borel--Moore homology of stacks will not satisfy these axioms. Another axiom is that the (co)homology of a derived Artin $\mathbb{C}$-stack $\mathcal{X}$ is the same as the (co)homology of its un-derived truncation $t_0(\mathcal{X}) \in \text{Ob}(\text{HArt}_\mathbb{C})$. Therefore there is no real distinction between (co)homology theories of derived Artin stacks and (co)homology theories of higher Artin stacks. A {\it (co)homology theory of higher $\C$-stacks with coefficients in $R$} is a collection of covariant functors $H_i(-) : Ho({\rm HArt}_\C) \ra R{\rm -mod}$ and contravariant functors $H^i(-) : Ho({\rm HArt}_\C) \ra R{\rm -mod}$ satisfying \cite[3.2.1]{Joy2}. If these functors satisfy every axioms except the one requiring $H^\ast({\rm Spec}(\C)) \cong H^\ast_{\rm sing}(\ast)$ then $H_i(-),H^i(-) : Ho({\rm HArt}_\C) \ra R{\rm -mod}$ is said to be a {\it generalized} (co)homology theory.

One can use simplicial presheaves to model the homotopy theory of higher $\mathbb{C}$-stacks. Given a finite type affine $\mathbb{C}$-scheme $U$, there is a simplicial set $U^\text{an}$ which is the singular complex of the underlying complex analytic space of $U$. Taking the left Kan extension along the simplicial Yoneda embedding yields a functor $(-)^{\rm Betti} : sPr(\text{Aff}_\mathbb{C}) \to \text{sSet}$. 
This functor is called the {\it Betti realization} of a simplicial presheaf. The Betti realization was first defined by Simpson \cite{Sim1}.

\begin{ex}
\label{hom4ex1}
Let $G$ be an algebraic group acting on a scheme $X$. Then 
    $$[X/G]^{\rm Betti} \simeq (EG^{\rm an} \times X^{\rm an})/G^{\rm an}$$
\cite[\S 8]{Sim1}.
\end{ex}

\begin{dfn}
\label{hom4dfn2}
Let $\mathcal{X}$ be a higher Artin $\mathbb{C}$-stack. The \emph{Betti cohomology of $\mathcal{X}$ with coefficients in $R$} is defined to be
        $$H^\ast(\mathcal{X},R) := H^\ast(\mathcal{X}^{\rm Betti},R),$$
\noindent where $H : Ho(\text{sSet}) \to Ho(R\text{-mod})$ denotes the $R$-coefficient singular cohomology of topological spaces. The Betti homology of $\mathcal{X}$ with coefficients in $R$ is defined similarly.
\end{dfn}      

We now recall the definition of the class of stacks whose cohomology can be given the structure of a graded vertex algebra\footnote{Joyce also defines vertex algebras on the homology of moduli stacks of objects in both abelian categories and triangulated categories. We work only with the triangulated category version of this theory.}.

\begin{dfn}
\label{hom4dfn3}
Let $\mathcal{A}$ be a saturated\footnote{A dg-category is {\it saturated} if it is smooth, proper, and triangulated \cite[Def.~2.4]{ToVa}. If $X$ is a smooth and proper variety, then Perf$(X)$ is saturated.} $\mathbb{C}$-linear dg-category and let $\mathcal{M}$ denote the moduli stack of objects in $\mathcal{A}$. Define stack morphisms $\Psi_\alpha : [\ast/\mathbb{G}_m] \times \mathcal{M}_\alpha \to \mathcal{M}_\alpha$ by $(\ast,E) \mapsto E$ on $\mathbb{C}$-points and by $(t,f) \mapsto t\cdot \text{Id}_{\mathcal{M}} \circ f$ on isotropy groups. Let $\Phi_{\alpha,\beta}: \mathcal{M}_\alpha \times \mathcal{M}_\beta \to \mathcal{M}_{\alpha + \beta}$ denote the stack morphism induced by the direct sum of objects in $\mathcal{A}$. Then a \emph{singular ring structure on $\M$} is the data of
\begin{itemize}
	 \item[1.] a quotient $K(\mathcal{A})$ of the Grothendieck group $K_0(\mathcal{A})$ of $\mathcal{A}$,
	 \item[2.] a symmetric bi-additive form $\chi : K(\mathcal{A}) \times K(\mathcal{A}) \to \mathbb{Z}$,
    \item[3.] signs $\epsilon_{\alpha,\beta} \in \{\pm 1\}$ for every $\alpha,\beta \in K(\mathcal{A})$, and 
    \item[4.] a perfect complex $\Theta^\bullet \in \text{Perf}(\mathcal{M} \times \mathcal{M})$ 
\end{itemize}
such that
\begin{itemize}
    \item[a.] $\mathcal{M}(\mathbb{C}) \to K(\mathcal
     A)$, $E \mapsto [E]$ is locally constant, giving a decomposition $\mathcal{M} = \coprod_{\alpha \in K(\mathcal{A})} \mathcal{M}_\alpha$ of open and closed sub-stacks $\mathcal{M}_\alpha \subset \mathcal{M}$ of objects of class $\alpha$,
    \item[b.] the restriction $\Theta_{\alpha,\beta} := \Theta|_{\mathcal{M}_\alpha \times \mathcal{M}_\beta}$ has rank $\chi(\alpha,\beta)$,
    \item[c.] for all $\alpha, \beta, \gamma \in K(\mathcal{A})$ 
        \begin{align*}
            \epsilon_{\alpha,\beta} \cdot \epsilon_{\beta,\alpha} & = (-1)^{\chi(\alpha,\beta) + \chi(\alpha,\alpha)\chi(\beta,\beta)} \\
            \epsilon_{\alpha,\beta}\cdot \epsilon_{\alpha+\beta,\gamma} & = \epsilon_{\alpha,\beta+\gamma} \cdot \epsilon_{\beta,\gamma} \\
            \epsilon_{\alpha,0} & = \epsilon_{0,\alpha} = 1,
        \end{align*}
     \item[d.] let $\sigma_{\alpha,\beta} : \mathcal{M}_\alpha \times \mathcal{M}_\beta \to \mathcal{M}_\beta \times \mathcal{M}_\alpha$ denote exchange of factors, then
     	$$\sigma^\ast_{\alpha,\beta} (\Theta^\bullet_{\alpha,\beta}) \cong (\Theta^\bullet_{\alpha,\beta})^\vee[2n],$$
\noindent for some $n \in \mathbb{Z}$ ,       
        \item[e.] for all $\alpha,\beta,\gamma \in K(\mathcal{A})$ there are isomorphisms
        	\begin{align*}(\Phi_{\alpha,\beta} \times \text{Id}_{\mathcal{M}_\gamma})^\ast (\Theta^\bullet_{\alpha+\beta,\gamma}) & \cong \pi^\ast_{\mathcal{M}_\alpha \times \mathcal{M}_\gamma}(\Theta^\bullet_{\alpha,\gamma}) \oplus \pi^\ast_{\mathcal{M}_\beta \times \mathcal{M}_\gamma} (\Theta^\bullet_{\beta,\gamma}) \\
        	(\text{Id}_{\mathcal{M}_\alpha} \times \Phi_{\beta,\gamma})^\ast (\Theta^\bullet_{\alpha,\beta+\gamma}) & \cong \pi^\ast_{\mathcal{M}_\alpha \times \mathcal{M}_\beta}(\Theta^\bullet_{\alpha,\beta}) \oplus \pi^\ast_{\mathcal{M}_\alpha \times \mathcal{M}_\gamma} (\Theta^\bullet_{\alpha,\gamma}) \\
        	(\Psi_\alpha \times \text{Id}_{\mathcal{M}_\beta})^\ast(\Theta^\bullet_{\alpha,\beta}) & \cong \pi^\ast_{[\ast/\mathbb{G}_m]} (E_1) \otimes \pi^\ast_{\mathcal{M}_\alpha,\mathcal{M}_\beta} (\Theta^\bullet_{\alpha,\beta}) \\
            	(\pi_{\mathcal{M}_\alpha}, \Psi_\beta \circ \pi_{[\ast/\mathbb{G}_m \times \mathcal{M}_\beta]})^\ast (\Theta^\bullet_{\alpha,\beta}) & \cong \pi^\ast_{[\ast/\mathbb{G}_m]} (E_{-1}) \otimes \pi^\ast_{\mathcal{M}_\alpha,\mathcal{M}_\beta} (\Theta^\bullet_{\alpha,\beta}).   
 	\end{align*}
\end{itemize} 
\end{dfn}

\noindent By abuse of language we say that a moduli stack $\M$ {\it is} a singular ring when we have in mind a particular singular ring $(\M, K(\mathcal{A}), \Th, \chi, \{\ep_{\al,\be\}_{\al,\be  \in K(\mathcal{A})}})$. The reason for the term ``singular ring" is explained in the introduction.

\begin{thm}[{see \cite[Thm.~4.11]{Joy2}}]
\label{hom4thm4}
Let  $(\M, K(\mathcal{A}), \Th, \chi, \{\ep_{\al,\be\}_{\al,\be  \in K(\mathcal{A})}})$ be a singular ring. Define a new grading by $$\hat{H}_i(\mathcal{M}_\alpha) := \hat{H}_{i-\chi(\alpha,\alpha)}(\mathcal{M}_\alpha).$$ \noindent Then $\hat{H}_\ast(\mathcal{M})$ is a graded vertex algebra over $R$. 
\end{thm}

This graded vertex algebra structure is written out explicitly as follows. The stack $[\ast / \mathbb{G}_m]$ acts on $\M$ by scaling stabilizers. Taking rational homology gives an operator $\mathcal{D}(z) : H_\ast(\M) \ra H_\ast(\M)[[z]]$.

Then the formula
\begin{equation}
\label{hom4eqn1}
\begin{split}
    Y(v,z)w = \epsilon_{\alpha,\beta} (-1)^{a \chi(\beta,\beta)} \Phi_\ast(\mathcal{D}(z) \boxtimes {\rm id})  [v \boxtimes w \cap  z^{\chi(\alpha,\beta)} \sum_{i \geq 1} z^{-i} c_i(\Theta^\bullet_{\alpha,\beta})].
\end{split}
\end{equation}
defines a state-to-field correspondence. Over a $\Q$-algebra, (\ref{hom4eqn1}) can be rewritten as
\begin{equation}
\label{hom4eqn2}
\begin{split}
    Y(v,z)w = \epsilon_{\alpha,\beta} (-1)^{a \chi(\beta,\beta)} \Phi_\ast(\mathcal{D}(z) \boxtimes {\rm id})  [v \boxtimes w \cap {\rm exp}(\sum_{i \geq 1} (-1)^{i-1}(i-1)!z^{-j} \\ \ch_i(\Th^\bu_{\al,\be}))]].
\end{split}
\end{equation}

We now compute the rational Betti homology of the higher $\C$-stack $\M$ of objects in ${\rm Perf}(X)$, where $X$ is a smooth complex projective variety in class D.

Note that the Betti realization of $\M$ is an H-space. The Milnor--Moore theorem therefore implies that, if the identity component $\M_0$ is finite type, the Hopf algebra $H_\ast(\M_0,\Q)$ is a free commutative-graded algebra on its primitive elements. This essentially reduces the problem to identifying the primitive elements of $H_\ast(\M_0, \Q)$ more explicitly. To this end, we begin with a model $\tilde{{\bf K}}^{\rm sst}({\rm Perf}(X))$, due to Blanc \cite{Bl}, of the homotopy type of the connective spectrum determined by the group-like $E_\infty$-space $\M^{\rm Betti}$. Antieu--Heller identify the homotopy type of Blanc's model with the connective semi-topological K-theory spectrum $K^{\rm sst}(X)$ of $X$ itself \cite{AnHe}. This implies that the $0^{\rm th}$ space of the spectrum $K^{\rm sst}(X)$ is equivalent, as an infinite loop space, to $\M^{\rm Betti}$. The reader will notice that we care only about the H-space structure on $\Om^\infty K^{\rm sst}(X)$, not the full infinite loop space structure.  

More explicitly, in \cite{Bl} Blanc introduces a functor $\tilde{\textbf{K}}^\text{sst} : \text{dgCat}_\C \to {\rm Sp}$ called the \emph{connective semi-topological K-theory of complex non-commutative spaces} \cite[Def.~4.1]{Bl}. This functor is defined as follows: for a $\C$-dg-category $\A$ there is a spectral presheaf $\tilde{\textbf{K}}(\A) : {\rm Aff}_\C^{\rm op} \ra {\rm Sp}$ such that $\tilde{\textbf{K}}(\A)({\rm Spec}(B)) \simeq \tilde{K}(\A \otimes^\mathbb{L}_\C B)$, where $\tilde{K}(\A \otimes^\mathbb{L}_\C B)$ denotes the connective K-theory spectrum of the $\C$-dg-category $\A \otimes^\mathbb{L}_\C B$. The \emph{connective semi-topological K-theory} of $\A$ is then defined to be the spectral realization (see Blanc \cite[\S.~3.4]{Bl}) of $\tilde{K}(\A)$.
  
\begin{prop}[{see  Antieu--Heller \cite[Thm.~2.3]{AnHe}, Blanc \cite[Thm.~4.21]{Bl}}]
\label{BAH}
Let $\mathcal{A}$ be a saturated $\mathbb{C}$-linear dg-category and let $\mathcal{M}_\mathcal{A}$ denote the moduli stack of objects in $\mathcal{A}.$ Then there is an equivalence
	$$\Om^\infty \tilde{{\bf K}}^\text{sst}(\mathcal{A}) \simeq \mathcal{M}_\mathcal{A}^{\rm Betti}$$
\noindent of infinite loop spaces which is canonical up to homotopy. When $\mathcal{A} = {\rm Perf}(X)$ there is a further equivalence 
 $$\Om^\infty \tilde{\textbf{K}}^\text{sst}(\text{Perf}(X)) \simeq \Om^\infty K^\text{sst}(X).$$
In particular, the set of connected components of $\M^{\rm Betti}_{{\rm Perf}(X)}$ is given by $K_\text{sst}^0(X).$ 
\end{prop}

It is also worth pointing out that the semi-topological K-theory of saturated dg-categories was previously defined by Bertrand To\"en \cite{To3}. To\"en actually {\it defines} the semi-topological K-theory space of a saturated dg-category $\mathcal{A}$ to be the Betti realization the moduli stack of objects in $\mathcal{A}$ (see also Kaledin \cite[\S.~8]{Kal}).

\noindent In particular, the set of connected components of $\M$ is given by $K_\text{sst}^0(X).$ The $0^{\rm th}$ semi-topological K-group $K_{\rm sst}^0(X)$ is isomorphic to the quotient of the Grothendieck group $K^0({\rm Vect}(X))$ by the relation of algebraic equivalence \cite{FrWa1}.

\begin{lem}
\label{hom5lem2}
Let $\alpha \in K_\text{sst}^0(X)$ and let $\M_\alpha \subset \M$ denote the substack of perfect complexes of coherent sheaves on $X$ of class $\alpha$. Then there is an $\mathbb{A}^1$-homotopy equivalence $\M_\alpha \simeq \M_0.$ 
\end{lem}
\begin{proof}
Let $\mathcal{E}^\bullet$ be a perfect complex with $\alpha = [\mathcal{E}^\bullet] \in K_\text{sst}^0(X)$. Define a morphism $\Phi_{\mathcal{E}^\bullet} : \M_0 \to \M_\alpha$ by $\mathcal{F}^\bullet \mapsto \mathcal{E}^\bullet \oplus \mathcal{F}^\bullet.$ Similarly, define a morphism $\Psi_{\mathcal{E}^\bullet} : \M_\alpha \to \M_0$ by $\mathcal{F}^\bullet \mapsto \mathcal{E}^\bullet[1] \oplus \mathcal{F}^\bullet.$ We claim that $\Psi_{\mathcal{E}^\bullet} \circ \Phi_{\mathcal{E}^\bullet} \cong \text{Id}_{M_0}.$ Define $H : \mathbb{A}^1 \times \M_0 \to \M_0$ by 
\begin{equation} 
\label{hom5eqn1}
	(t,\mathcal{F}^\bullet) \mapsto \mathcal{F}^\bullet \oplus \text{Cone}(t \cdot \text{Id}_{\mathcal{E}^\bullet}).
\end{equation}
Then, for $\mathcal{F}^\bullet \in {\rm ob}(\text{Perf}(X))$, $$H(0,\mathcal{F}^\bullet) = \mathcal{F}^\bullet \oplus \mathcal{E}^\bullet \oplus \mathcal{E}^\bullet[1] = \Psi_{\mathcal{E}^\bullet} \circ \Phi_{\mathcal{E}^\bullet} (\mathcal{F}^\bullet)$$
and, for $t \in \mathbb{A}^1 / \{0\}$,
$$H(t, \mathcal{F}^\bullet) = \mathcal{F}^\bullet = \text{Id}_{\M_0}(\mathcal{F}^\bullet).$$
Thus $\Psi_{\mathcal{E}^\bullet}$ is a left homotopy inverse of $\Phi_{\mathcal{E}^\bullet}.$ For the other direction, one may define a homotopy $H' : \mathbb{A}^1 \times \M_\alpha \to \M_\alpha$ between $\Phi_{\mathcal{E}^\bullet} \circ \Psi_{\mathcal{E}^\bullet}$ and Id$_{\M_\alpha}$ again by the formula (\ref{hom5eqn1}).
\end{proof}

\begin{thm}
\label{hom5thm4}
Suppose that the identity component $\M_0$ of $\M$ is finite type. Then there is a natural isomorphism of $\mathbb{Q}$-Hopf algebras 
\begin{equation}
\label{hom5eqn6}
    H_\ast(\M,\mathbb{Q}) \cong \mathbb{Q}[K_\text{sst}^0(X)] \otimes {\rm SSym}_\Q[\bigoplus_{i > 0} K^i_{\rm sst}(X) ].
\end{equation}
\end{thm}
\begin{proof}
The Milnor--Moore theorem gives an
	$$H_\ast(\M_0,\Q) \cong {\rm SSym}_\Q[\bigoplus_{i \geq 1} K_{\rm sst}^i(X)]$$
of graded $\Q$-Hopf algebras. By this, Lemma \ref{hom5lem2}, and the K\"unneth formula 
\begin{equation*}
	H_\ast(\M, \Q) \cong H_\ast(\pi_0(\M) \times \M_0, \Q)
		               \cong \mathbb{Q}[K_\text{sst}^0(X)] \otimes {\rm SSym}_\Q[\bigoplus_{i \geq 1} K_{\rm sst}^i(X)].  \qedhere
\end{equation*}
\end{proof}

Computing semi-topological K-theory is, in general, not an easy task. There is however a certain class of varieties for which computing semi-topological K-theory is not hard.

\begin{dfn}
\label{hom5dfn5}
A smooth complex variety $V$ is said to be {\it in class D} if the natural map $\Om^\infty K^{\rm sst}(V) \ra \Om^\infty K^{\rm top}(V^{\rm an})$ induces an isomorphism $K_{\rm sst}^i(V) \ra K_{\rm top}^i(V^{\rm an})$ for all $i \geq 1$ and a monomorphism $K^0_{\rm sst}(V) \hookrightarrow K^0_{\rm top}(V^{\rm an})$.
\end{dfn}

\begin{rem}
\label{hom5rem6}
Our terminology is motivated by the following: Friedlander--Haesemeyer--Walker say that a variety $V$ is {\it in class C} if the refined cycle maps $L_tH_n(X) \ra \tilde{W}_{-2t} H_n^{\rm BM}(X)$ \cite[Def.~5.8]{FrHaWa} are isomorphisms for all $t$ and $n$. The condition of being in class C is much stronger than what we want. For example, a smooth surface $S$ will not be in class C unless all of $H^2(S^{\rm an})$ is algebraic. This would exclude many interesting surfaces, such as K3 surfaces. If a variety is in class C then $K^i_{\rm sst}(V) \ra K^i_{\rm top}(V^{\rm an})$ is an isomorphism for $i \geq {\rm dim}_\C V - 1$ and injective for $i = {\rm dim}_\C V - 2$; one could call this property being in class E. If so, then {\it all} surfaces are in class E \cite[Thm.~3.7]{FrHaWa} and class C $\subset$ class D $\subset$ class E. 
\end{rem}

\noindent Fortunately, many varieties are in class D. Examples of projective varieties in class D are curves, surfaces, toric varieties, flag varieties, and rational 3- and 4-folds \cite[Thm.~6.18, Prop.~6.19]{FrHaWa}. This is a non-exhaustive list. For example, some more examples of degenerate 3- and 4-folds in class D were computed recently by Voineagu \cite{Voin}. The kernel of $K^0_{\rm sst}(X) \otimes \Q \ra K^0_{\rm top}(X^{\rm an}) \otimes \Q$ is isomorphic to the rational Griffiths group \cite[Ex.~1.5]{FrWa2}. Voison proved that general Calabi--Yau 3-folds have infinitely generated rational Griffiths group so that they cannot be in class D \cite[Thm.~4]{Vois}. The following lemma justifies our choice to focus on varieties in class D.

\begin{lem} 
\label{hom5lem7}
If $X$ is in class D then for all $\al \in K_{\rm sst}^0(X)$ the K-theory comparison map induces a homotopy equivalence $\M_\al^{\rm Betti} \simeq \Map_{C^0}(X^{\rm an}, BU \times \Z)_\al$.
\end{lem}
\begin{proof}
By definition of class D, each connected component $\Xi_\al : \Om^\infty K^{\rm sst}(X)_\al \ra  \Om^\infty K^{\rm top}(X^{\rm an})_\al$ of the natural K-theory comparison map  is a weak homotopy equivalence. Proposition 5.1 then gives a weak homotopy equivalence
	\begin{equation}
	\label{hom5eqn8}
		\M^{\rm Betti}_\al \simeq \Om^\infty K^{\rm sst}(X)_\al \stackrel{\Xi}{\longrightarrow} \Om^\infty K^{\rm top}(X^{\rm an})_\al \simeq {\rm Map}_{C^0}(X^{\rm an}, BU \times \Z)_\al.
	\end{equation}
As $X^{\rm an}$ is a compact metric space and $BU \times \Z$ has the homotopy type of a countable CW complex, $ {\rm Map}_{C^0}(X^{\rm an}, BU \times \Z)_\al$ has the homotopy type of a CW complex \cite[Cor.~2]{Mil}. Because $\M^{\rm Betti}_\al$ is the realization of a simplicial set it is a CW complex too. Whitehead's theorem then lifts the weak homotopy equivalence (\ref{hom5eqn8}) to a homotopy equivalence.
\end{proof}

\noindent The following lemma allows us to define Chern classes and Betti realizations of perfect complexes on higher $\C$-stacks.

\begin{lem}[{see Blanc \cite[Thm.~4.5]{Bl}}]
\label{hom5prop9}
Let ${\rm Perf}_\C$ denote the $\C$-stack of perfect complexes of complex vector spaces. There is an equivalence of infinite loop spaces
    $${\rm Perf}_\C^{\rm Betti} \simeq BU \times \Z.$$
\end{lem}

\begin{dfn}
\label{hom5dfn10}
Let $\X$ be a higher $\C$-stack and let $\E^\bu$ be a rank $r$ perfect complex on $\X$. Then there is a map $\phi_{\E^\bu} : \X \ra {\rm Perf}_\C^r$ classifying $\E^\bu$. By Lemma \ref{hom5lem7}, taking  cohomology then gives a homomorphism $$H^\ast(\E^\bu) : H^\ast({\rm Perf^r}_\C) \cong  H^\ast(BU) \cong \Z[[c_1,c_2,\dots]] \ra H^\ast(\X).$$ The {\it $i^{\rm th}$ Chern class of $\E^\bu$} is $c_i(\E^\bu) := H^\ast(\E^\bu)(c_i)$. 
\end{dfn}

\begin{dfn}
\label{hom5dfn11}
Let $\X, \E^\bu,$ and $ \phi_{\E^\bu}$ be as in Definition \ref{hom5dfn10}. Then functoriality of the Betti realization gives a continuous map $(\uE)^{\rm Betti} : \X^{\rm Betti} \ra {\rm Perf}_\C^{\rm Betti} \simeq BU \times \Z$ called the {\it Betti realization} of $\uE$.
\end{dfn}

\noindent The following Proposition makes it possible to compute the Chern classes of the universal complex over $X \times \M$ when $X$ is in class D.

\begin{prop}
\label{hom5prop12}
Let $X$ be in class D and $\al \in K^0_{\rm sst}(X)$. Let $\uE_\al$ be the universal perfect complex over $X \times \M_\al$ and let $\mathcal{E}_\al : X^{\rm an} \times \Map_{C^0}(X^{\rm an},BU)_\al \ra BU$ denote the evaluation map. Then there is a homotopy $$(\uE_\al)^{\rm Betti} \simeq \E_\al : X^{\rm an} \times \Map_{C^0}(X^{\rm an},BU)_\al \ra BU$$
In particular, the image of $c_i(\uE_\al)$ under the isomorphism $H^{2i}(X \times \M_\al) \cong H^{2i}(X^{\rm an} \times \Map_{C^0}(X^{\rm an},BU)_\al)$ equals $c_i([\mathcal{E}_\al])$ for all $i \geq 0$.
\end{prop}
\begin{proof}
The evaluation map $X \times {\rm Map}_{{\rm HSt}_\C}(X, {\rm Perf}_\C)_\al \ra {\rm Perf}_\C$ classifies $\uE_\al$. Taking Betti realizations gives a map
	$$(ev_\al)^{\rm Betti}: X^{\rm an} \times \M^{\rm Betti}_\al \ra BU \times \Z$$
by Lemma \ref{hom5prop9}. Exponentiating gives a continuous map $\Ga_\al : \Om^\infty K^{\rm sst}(X)_\al \ra \Om^\infty K^{\rm top}(X^{\rm an})_\al$. Write $$\Ga := \coprod_{\al \in K^0_{\rm sst}(X)} \Ga_\al : \Om^\infty K^{\rm sst}(X) \ra \Om^\infty K^{\rm top}(X^{\rm an}).$$ Note that both $\Ga$ and the natural K-theory comparison map $\Xi : \Om^\infty K^{\rm sst}(X) \ra \Om^\infty K^{\rm top}(X^{\rm an})$ make the homotopy-theoretic group completion diagram 
\begin{center}
    \begin{tikzcd}
    	{\rm Map}_{\rm alg}(X, \Gr)^{\rm an}  \arrow{r} \arrow{d}  & {\rm Map}_{C^0}(X^{\rm an}, \coprod_{n \geq 0} BU(n)) \arrow{d} \\
	\Om^\infty K^{\rm sst}(X) \arrow[swap]{r}{\Ga, \Xi} & \Om^\infty K^{\rm top}(X^{\rm an})
    \end{tikzcd}
\end{center}
homotopy commute. By the weak universal property of homotopy theoretic group completions, $\Ga$ is weakly homotopic to $\Xi$ \cite[Prop.~1.2]{CCMT}. Therefore, the restrictions of $\Ga,\Xi$ along any map $S^n \ra \Om^\infty K^{\rm sst}(X)$ are homotopic. In particular, $\Ga$ and $\Xi$ induce the same maps on homotopy groups. Because $X$ is in class $D$, $\Xi_\al$ is a homotopy equivalence for all $\al \in K^0_{\rm sst}(X)$. This now implies that $\Ga_\al$ is a homotopy equivalence for all  $\al \in K^0_{\rm sst}(X)$. Because $\Ga_\al$ makes the diagram
\begin{center}
	\begin{tikzcd}
		X^{\rm an} \times \M_\al^{\rm Betti} \arrow{r}{(\uE_\al)^{\rm Betti}} \arrow[swap]{d}{1_{X^{\rm an} \times \Ga_\al}} & BU \\
		X^{\rm an} \times {\rm Map}_{C^0}(X^{\rm an}, BU)_\al \arrow[swap]{ur}{\E_\al}
	\end{tikzcd}
\end{center}
homotopy commute, $(\uE_\al)^{\rm Betti}$ is homotopic to $\E_\al$ for all $\al \in K^0_{\rm sst}(X)$.
\end{proof}

\begin{thm}
\label{hom5thm13}
Let $X$ be in class D and let $\al \in K_{\rm sst}^0(X)$. Then $H^\ast(\M_\al,\Q)$ is freely generated as a commutative-graded $\Q$-algebra by the K\"unneth components of Chern classes of the universal complex $\uE_\al$ over $X \times \M_\al$.
\end{thm}
\begin{proof}
By Proposition \ref{hom5prop12}, it suffices to show that $H^\ast({\rm Map}_{C^0}(X^{\rm an}, BU)_\al, \Q)$ is freely generated by K\"unneth components of Chern classes of $[\E_\al]$.
Recall that $c_k(\E_\al)$ equals the image of $c_k \in \Q[[c_1,c_2,\dots]] \cong H^\ast(BU,\Q)$ under the rational cohomology of the evaluation map $ev_\al : X^{\rm an} \times {\rm Map}_{C^0}(X^{\rm an}, BU)_\al \ra BU$. 

Because BU is a finite type simply-connected H-space, Proposition \ref{hom2prop11} gives that $({\rm Sym}_\Q(\pi_\ast(BU)^\vee), 0)$ is the minimal model of $BU$. Because $X^{\rm an}$ is K\"ahler, it is a formal space \cite[p.~270]{DGMS}. Therefore, we may choose $(H^\ast(X^{\rm an}, \Q), 0)$ as our finite-dimensional model for $X^{\rm an}$. Let $\{v_1,\dots,v_n \}$ be an additive $\Q$-basis for $H^\ast(X^{\rm an}, \Q)$ and let $\{v_1^\vee,\dots,v_n^\vee \}$ be a dual $\Q$-basis for $H_\ast(X^{\rm an}, \Q)$. Then, by Proposition \ref{hom2prop13}, the the rational cohomology of the evaluation map $$H^\ast(ev_\al,\Q) : H^\ast(BU,\Q) \ra H^\ast(X^{\rm an}, \Q) \otimes {\rm SSym}_\Q[\pi_\ast(BU)^\vee \otimes H_\ast(X^{\rm an}, \Q)]$$ is given by $$c_k \mapsto \sum_i v_i \otimes (v_i^\vee \otimes c_k).$$ From this, one has that the generators of ${\rm SSym}_\Q[\pi_\ast(BU)^\vee \otimes H_\ast(X^{\rm an}, \Q)]$ can all be written as slant products of Chern classes of $[\E_\al]$ with rational homology classes.
\end{proof}

\noindent Because the leading coefficients of the universal Chern character polynomials are non-zero, one can also regard the cohomology $H^\ast(\M_\al,\Q)$ as being freely generated as a commutative-graded $\Q$-algebra by the K\"unneth components of Chern {\it characters} of $\uE_\al$.  

Fix, for the remainder of this document, a basis $Q = \{v_1,\dots, v_r\}$ of the K-theory $K^0_{\rm top}(X^{\rm an})_\Q \oplus K^1_{\rm top}(X^{\rm an})_\Q$ of $X^{\rm an}$ and a dual basis $Q^\vee = \{v_1^\vee, \dots, v_r^\vee \}$ of $(K^0_{\rm top}(X^{\rm an})_\Q \oplus K^1_{\rm top}(X^{\rm an})_\Q)^\vee$. By the above Corollary, if $X$ is in class D, then for all $\al \in K^0_{\rm sst}(X)$ there is a canonical isomorphism of graded $\Q$-algebras
    \begin{equation}
        H^\ast(\M_\al,\Q) \cong {\rm SSym}[[\mu_{\al,v,i} : v \in Q, i \geq 1]],
    \end{equation}
given by $\ch_i([\uE_\al]/v^\vee) \mapsto \mu_{\al, v, i}$.

\section{Field calculations}

We explicitly compute Joyce's fields in the case that $X$ is in class D. Lel $(-)^\vee : H^\ast(X^{\rm an},\Q) \ra H^\ast(X^{\rm an},\Q)$ denote the involution 
	$$v^\vee = \begin{cases} (-1)^{{\rm deg}(v)/2} v, & 2 | {\rm deg}(v)   \\
						(-1)^{({\rm deg}(v)-1)/2} v, & 2 \nmid {\rm deg}(v).
	\end{cases}$$
By abuse of notation, we also write $(-)^\vee$ for the involution on K-theory induced by taking duals of bundles.	

We define a super-symmetric bilinear form on $K^0_{\rm top}(X^{\rm an}) \oplus K^1_{\rm top}(X^{\rm an})$ by
    $$\chi(v,w) = \int_{X^{\rm an}} \ch(v)^\vee \cdot \ch(w) \cdot {\rm Td}(X^{\rm an})$$
if $X$ is $2n$-Calabi--Yau for some $n \geq 1$ and by
    $$\chi_{\rm sym}(v,w) = \int_{X^{\rm an}} \big ( \ch(v)^\vee \cdot \ch(w) + \ch(w)^\vee \cdot \ch(v) \big ) \cdot {\rm Td}(X^{\rm an}) $$
if $X$ is not $2n$-Calabi--Yau for any $n \geq 1$. 

Let $\uE$ denote the universal complex over $X \times \M$ and, for $\al \in K^0_{\rm sst}(X)$, let $\uE_\al$ denote the universal complex over $X \times \M_\al$. Let $\E$ the evaluation map for Map$_{C^0}(X, BU \times \Z)$ and let $\E_\al$ denote the evaluation map for Map$_{C^0}(X^{\rm an}, BU \times \Z)_\al$. For $\al,\be \in K^0_{\rm sst}(X)$ define complexes $\cExt_{\al,\be} \in {\rm Perf}(\M_\al \times \M_\be)$ by
    $$\cExt_{\al,\be} := \R\pi_\ast(\pi_1^\ast(\uE_\al)^\vee \otimes^\mathbb{L} \pi_2^\ast(\uE_\be)).$$Then for all $m \in \Z$
    $$\si^\ast_{\al,\be} ((\cExt_{\al,\be})^\vee \oplus \si^\ast_{\al,\be}\cExt_{\be,\al}[2m]) \cong ((\cExt_{\al,\be})^\vee \oplus \si^\ast_{\al,\be}\cExt_{\be,\al}[2m])^\vee [2m].$$
and if $X$ happens to be  $2n$-Calabi--Yau then
$$\si_{\al,\be}^\ast ((\cExt_{\be,\al})^\vee) \cong  \cExt_{\al,\be}[2n].$$ 

This allows us to make $\hat{H}_\ast(\M,\Q)$ into a graded vertex -algebra. If $X$ is $2n$-Calabi--Yau then let $Y(-,z) : \hat{H}_\ast(\M,\Q) \ra \mathcal{F}(\hat{H}_\ast(\M,\Q))$  denote the linear map defined by (4.3) taking $\Th^\bu = (\cExt)^\vee$, $\chi$ to be the restriction of the Euler form to $K^0_{\rm sst}(X)$, and taking $\{\ep_{\al,\be}\}_{\al,\be \in K^0_{\rm sst}(X)}$ any solution of (\ref{hom3eqn3})-(\ref{hom3eqn5}) . If $X$ is not $2n$-Calabi--Yau then let $Y(-,z) : \hat{H}_\ast(\M,\Q) \ra \mathcal{F}(\hat{H}_\ast(\M,\Q))$ denote the linear map defined by (\ref{hom4eqn1}) taking $\Th^\bu = (\cExt)^\vee \oplus \si^\ast(\cExt)$, $\chi$ to be the restriction of the symmetrised Euler form to $K^0_{\rm sst}(X)$, and $\ep_{\al,\be} = (-1)^{\chi(\al,\be)}$.

To get an explicit formula for $Y(-,z)$ one has to calculate the $E$-Chern classes of $\Th^\bu$. We only know how to do this when $X$ is in class D. For brevity we write $\mathcal{U} = \pi^\ast(\E)^\vee \otimes \pi^\ast(\E)$.

\begin{lem}
\label{hom6lem2}
Let $X$ be in class D and let $\al,\be \in K_{\rm sst}^0(X)$. Then, for all $i \geq 0$ 
    $$c_i((\cExt_{\al, \be})^{\rm Betti}) = c_i(\pi^{KU}_!(\mathcal{U}_{\al,\be})).$$
\end{lem}
\begin{proof}
One can describe $\pi^{KU}_!(\pi^\ast_\al (\E_\al)^\vee \otimes \pi^\ast_\be(\E_\be))$ using twisted elliptic operators. Consider the elliptic operator
	$$D := \overline{\partial} + \overline{\partial}^* : C^\infty(\Lambda^{0,2\ast}T^\ast X^{\rm an}) \ra C^\infty(\Lambda^{0,2\ast+1}T^\ast X^{\rm an})$$ 
on $X^{\rm an}$. Given complex vector bundles $P,Q \ra X^{\rm an}$ we can choose connections $\nabla_P, \nabla_Q$ on them and we can write down a Fredholm operator $D^{\nabla_{\overline{P} \times Q}} : C^\infty(\Lambda^{0,2\ast}T^\ast X^{\rm an} \otimes \overline{P} \otimes Q) \ra C^\infty(\Lambda^{0,2\ast+1}T^\ast X^{\rm an} \otimes \overline{P} \otimes Q)$ as in \cite[Def.~2.20]{JTU}. Then $(P,Q) \mapsto D^{\nabla_{\overline{P} \times Q}}$ gives a map $$D^U : \coprod_{n \geq 0} {\rm Map}_{C^0}(X^{\rm an}, BU(n)) \times \coprod_{n \geq 0} {\rm Map}_{C^0}(X^{\rm an}, BU(n)) \ra {\rm Fred}(H)$$ whose homotopy class is independent of the choices of connections $\nabla_P, \nabla_Q$. By group-likeness of Fred$(H)$ and the weak universal property of homotopy-theoretic group completions there exists a weak H-map $$\overline{D}^U : \Om^\infty K^{\rm top}(X^{\rm an}) \times \Om^\infty K^{\rm top}(X^{\rm an}) \ra {\rm Fred}(H) \simeq BU \times \Z$$ such that the restriction of $\overline{D}^U$ along the completion map $$\coprod_{n \geq 0} {\rm Map}_{C^0} (X^{\rm an}, BU(n)) \ra \Om^\infty K^{\rm top}(X^{\rm an})$$ is weakly homotopic to $D^U$. By the families index theorem \cite[Thm.~3.1]{AtSi}, the weak homotopy class of $\overline{D}^U$ equals the weak homotopy class of $\pi^{KU}_!(\mathcal{U})$. As $X$ is in class D, it then suffices to show that the restrictions
\begin{center}
    \begin{tikzcd}
        \M^{\rm Betti} \times \M^{\rm Betti} \arrow[hookrightarrow]{r} & \Om^\infty K^{\rm top}(X^{\rm an}) \times \Om^\infty K^{\rm top}(X^{\rm an}) \arrow[shift left = 1.0ex]{r}{\overline{D}^U} \arrow[shift right = 1.0ex, swap]{r}{(\cExt)^{\rm Betti}} & BU \times \Z
    \end{tikzcd}
\end{center}
are weakly homotopic. If $C \subset \Map_{\rm alg}(X, \Gr)^{\rm an} \times \Map_{\rm alg}(X,\Gr)^{\rm an}$ is compact then $[(\cExt)^{\rm Betti}|_C] \in K^0_{\rm top}(C)$ equals the index bundle $[D^U|_C]$. In particular, the restrictions of $(\cExt)^{\rm Betti}$ and $\overline{D}^U$ to $\Map_{\rm alg}(X, \Gr)^{\rm an} \times \Map_{\rm alg}(X,\Gr)^{\rm an}$ are weakly homotopic. The claim then follows from Proposition \ref{BAH} and the fact that Fred$(H)$ is group-like.
\end{proof}

\begin{prop}
\label{hom6prop3}
Let $X$ be in class D. Then for all $\al,\be \in K^0_{\rm sst}(X)$, $i \geq 1$
    \begin{equation*}
    \begin{split}
        \ch_i((\cExt_{\al,\be})^\vee) = \sum_{\stackrel{j,k \geq 0; i = j+k}{v,w \in Q}} (-1)^k \chi(v,w) \mu_{\al,v,j} \boxtimes \mu_{\be,w,k}.
    \end{split}
    \end{equation*}  
\noindent and     
    \begin{equation*}
    \begin{split}
        \ch_i((\cExt_{\al,\be})^\vee \oplus (\si^\ast \cExt_{\be,\al})) = \sum_{\stackrel{j,k \geq 0; i = j+k}{v,w \in Q}} (-1)^k \chi_{\rm sym}(v,w) \mu_{\al,v,j} \boxtimes \mu_{\be,w,k}.
    \end{split}
    \end{equation*}  
\end{prop}
\begin{proof}
By Lemmas \ref{hom6lem2} and Dold's Atiyah--Hirzebruch--Riemann--Roch theorem \cite{Dold}, we compute 
    \begin{align*}
    \begin{split}
    	\ch((\cExt_{\al,\be})^\vee) & = (\int_{X^{\rm an}} \pi^\ast_\al(\ch(\E_\al))^\vee \cdot \pi^\ast_\be (\ch^E(\E_\be)) \cdot \pi^\ast_{\al,\be}{\rm Td}(X^{\rm an}))^\vee \\
	& = (\int_{X^{\rm an}} \pi^\ast_\al\ch(\sum_{v \in Q} v^\vee \boxtimes ([\E_\al] / v)^\vee) \cdot \pi^\ast_\be\ch(\sum_{w \in Q} w \boxtimes [\E_\be]/ w ) \cdot \\ \pi^\ast_{\al,\be}{\rm Td}(X^{\rm an}))^\vee \\
	& = (\sum_{v,w \in Q} \chi(v,w) \ch([\E_\al]/v)^\vee \ch([\E_\be]/w) )^\vee.
    \end{split}
    \end{align*}
So, for $i \geq 0$,
	$$\ch_i((\cExt_{\al,\be})^\vee) = \sum_{\stackrel{v,w \in Q}{i=j+k}} (-1)^k \chi(v,w) \mu_{\al,v,j} \boxtimes \mu_{\be,w,k}.$$ 
The  $(\cExt_{\al,\be})^\vee \oplus \si^\ast \cExt_{\be,\al}$  case is similar. 
\end{proof}

For $\al \in K^0_{\rm sst}(X)$, consider the $\Q$-algebra ${\rm SSym}_\Q[u_{\al,v,i} : v \in Q, i \geq 1]$. Define a pairing
    $${\rm SSym}_\Q[[\mu_{\al,v,i} : v \in Q, i \geq 1]] \times {\rm SSym}_\Q[u_{\al,w,i} : w \in Q, i \geq 1] \longrightarrow \Q$$
by     
\begin{equation}
\label{hom6eqn2}
    (\prod_{v \in Q, i \geq 1} \mu_{\al,v,i}^{m_{v,i}}) \cdot (\prod_{v \in Q, i \geq 1} u_{\al,v,i}^{n_{v,i}}) = \begin{cases}
        \prod_{v \in Q, i \geq 1} \frac{m_{v,i}!}{((i-1)!)^{m_{v,i}}}, & m_{v_i,i} = n_{v_i,i}, \forall i \\
        0, & {\rm otherwise.}
    \end{cases}
\end{equation}

\noindent We identify $H_\ast(\M_\al,\Q)$ with ${\rm SSym}_\Q[u_{\al,w,i}: i \geq 1, w \in Q]$ using (\ref{hom6eqn2}).

\begin{lem}
\label{hom6lem4}
Let $X$ in class $D$. Then for all $\al, \be \in K^0_{\rm sst}(X)$
    $$H_\ast(\Phi_{\al,\be})[(\prod_{v \in Q, i \geq 1} u^{m_{v,i}}_{\al,v,i}) \boxtimes (\prod_{v \in Q, i \geq 1} u^{n_{v,i}}_{\be,v,i})] = \prod_{v \in Q, i \geq 1} u^{m_{v,i} + n_{v,i}}_{\al+\be,v,i}.$$
\end{lem}
\begin{proof}
Pulling $\uE_{\al+\be}$ along the map ${\rm Id} \times \Phi_{\al,\be} : X \times \M_\al \times \M_\be \rightarrow X \times \M_{\al + \be}$ gives
    $$({\rm Id}_X \times \Phi_{\al,\be})^\ast (\uE_{\al+\be}) \cong \pi^\ast_1(\uE_\al) \oplus \pi^\ast_2(\uE_\be).$$
Taking K-theory classes, slanting with $v^\vee \in Q^\vee$, and then taking $\ch_i$ gives
\begin{align*}
     H^\ast(\Phi_{\al,\be})(\mu_{\al+\be,v,i}) & = \mu_{\al,v,i} \boxtimes 1 + 1 \boxtimes \mu_{\be,v,i}
\end{align*}
so that 
\begin{equation}
\label{hom6eqn3}
    H^\ast(\Phi_{\al,\be})(\prod_{v \in Q, i \geq 1}\mu_{\al+\be,v,i}^{n_{v,i}}) = \prod_{v \in Q, i \geq 1} \sum_{0 \leq m_{v,i} \leq n_{v,i}} {n_{v,i} \choose m_{v,i}} \mu_{\al,v,i}^{m_{v,i}} \boxtimes \mu_{\be,v,i}^{n_{v,i}-m_{v,i}}.
\end{equation}
Under (\ref{hom6eqn2}), $(\prod_{v \in Q, i \geq 1} u^{m_{v,i}}_{\al,v,i}) \boxtimes (\prod_{v \in Q, i \geq 1} u^{n_{v,i}}_{\al,v,i})$ gets identified with the functional 
$$(\prod_{v \in Q, i \geq 1} \mu^{m'_{v,i}}_{\al,v,i}) \boxtimes (\prod_{v \in Q, i \geq 1} \mu^{n'_{v,i}}_{\be,v,i}) \mapsto 
\begin{cases}
    \prod_{v \in Q, i \geq 1} \frac{m_{v,i}!n_{v,i}!}{((i-1)!)^{m_{v,i} + n_{v,i}}}, & \begin{split} m'_{v,i} = m_{v,i}, \\ n'_{v,i} =  n_{v,i} \end{split} \\
    0, & \text{otherwise}
\end{cases}$$
so that, by (\ref{hom6eqn3}), $E_\ast(\Phi_{\al,\be})$ acts as
    \begin{align*}
    \begin{split}
        \prod_{v \in Q, i \geq 1} \mu_{\al+\be,v,i}^{\ell_{v,i}} 
        & \mapsto (\prod_{v \in Q, i \geq 1} u^{m_{v,i}}_{\al,v,i}) \boxtimes (\prod_{v \in Q, i \geq 1} u^{n_{v,i}}_{\al,v,i}) \big ( \prod_{v \in Q, i \geq 1} \sum_{0 \leq r_{v,i} \leq \ell_{v,i}} {\ell_{v,i} \choose r_{v,i}} \\ \mu_{\al,v,i}^{r_{v,i}} \boxtimes \mu_{\be,v,i}^{\ell_{v,i}-r_{v,i}} \big ) \\
        & = \begin{cases} \begin{split} \prod_{v \in Q, i \geq 1} {n_{v,i} + m_{v,i} \choose m_{v,i}} \frac{n_{v,i}! m_{v,i}!}{((i-1)!)^{n_{v,i} + m_{v,i}}}, & \ell_{v,i} = n_{v,i} +  \\ m_{v,i}, \text{ for all } i \end{split} \\
        0, & \text{otherwise}
        \end{cases} \\
        & = \begin{cases} \begin{split} \prod_{v \in Q, i \geq 1} \frac{(m_{v,i} + n_{v,i})!}{((i-1)!)^{m_{v,i}+n_{v,i}}}, & \ell_{v,i} = n_{v,i} + m_{v,i}, \\ \text{ for all } i,  \end{split} \\
    0, & \text{otherwise.}
    \end{cases}
    \end{split}
    \end{align*}
This functional is represented by $\prod_{v \in Q, i \geq 1} u^{n_{v,i}+m_{v,i}}_{\al+\be,v,i}$.
\end{proof}

\begin{lem}
\label{hom6lem5}
Let $X$ be in class D and let $\al \in K^0_{\rm sst}(X)$. Then
    $$(\prod_{v \in Q, i \geq 1} u^{n_{v,i}}_{\al,v,i}) \cap (\prod_{v \in Q, i \geq 1} \mu^{m_{v,i}}_{\al,v,i}) = \begin{cases} \begin{split} \prod_{v \in Q, i \geq 1} \frac{n_{v,i}!}{(n_{v,i}-m_{v,i})!((i-1)!)^{m_{v,i}}}u^{n_{v,i}-m_{v,i}}_{\al,v,i}, &  \\ n_{v,i} \geq m_{v,i} \text{ for all } v,i \end{split} \\
    0, \hspace{14em} \text{otherwise.}
    \end{cases}$$
\end{lem}
\begin{proof}
The cap product is dual to the cup product under (\ref{hom6eqn2}). Therefore, the cap product $(\prod_{v \in Q, i \geq 1} u^{n_{v,i}}_{\al,v,i}) \cap (\prod_{v \in Q, i \geq 1} \mu^{m_{v,i}}_{\al,v,i})$ acts as
\begin{align*}
    \prod_{v \in Q, i \geq 1} \mu^{\ell_{v,i}}_{\al,v,i} & \mapsto u^{n_{v,i}}_{\al,v,i}(\prod_{v \in Q, i \geq 1} \mu_{\al,v,i}^{m_{v,i}} \cdot \prod_{v \in Q, i \geq 1} \mu_{\al,v,i}^{\ell_{v,i}}) \\
    & = \begin{cases}
                    \prod_{v \in Q, i \geq 1} \frac{n_{v,i}!}{((i-1)!)^{n_{v,i}}}, & \ell_{v,i} = n_{v,i} - m_{v,i} \\
                    0, & \text{otherwise}
        \end{cases}
\end{align*}
and $\prod_{v \in Q, i \geq 1} \frac{n_{v,i}!}{(n_{v,i}-m_{v,i})!((i-1)!)^{m_{v,i}}}u^{n_{v,i}-m_{v,i}}_{\al,v,i}(\prod_{v \in Q, i \geq 1} \mu^{\ell_{v,i}}_{\al,v,i})$ equals 
$\prod_{v \in Q, i \geq 1} \\ \frac{n_{v,i}!}{(n_{v,i}-m_{v,i})!((i-1)!)^{m_{v,i}}} \frac{(n_{v,i}-m_{v,i})!}{((i-1)!)^{n_{v,i}-m_{v,i}}}$ if $\ell_{v,i} = n_{v,i} + m_{v,i}$ for all $i$ and equals $0$ otherwise.
\end{proof}

\begin{lem}
\label{hom6lem6}
Let $X$ be in class D, $k \geq 0$, and $v \in Q$. Then
    $$H_\ast(\Psi_0)(t^i \boxtimes u_{0,v,1}) =  u_{0,v,i+1}.$$
\end{lem}
\begin{proof}
There is an isomorphism
    \begin{equation}
    \label{hom6eqn4}
    (\Psi_0 \times {\rm Id}_X)^\ast(\uE_0) \cong \pi^\ast_{[\ast/\mathbb{G}_m]}(E_1) \otimes \pi^\ast_{\M_0 \times X}(\uE_0),
    \end{equation}
where $E_1 \ra [\ast / \mathbb{G}_m]$ is the one-dimensional weight $1$ representation of $\mathbb{G}_m$.
Taking K-theory classes of (\ref{hom6eqn4}), slanting both sides by some $v^\vee \in Q^\vee$, and then taking $\ch^E_i$ gives
$$H^\ast(\Psi_0)(\mu_{0,v,i}) = \sum^i_{j=0} \frac{1}{j!} \tau^j \boxtimes \mu_{0,v,i-j},$$
where $\ch_j(E_1) = \frac{1}{j!} \tau^j$ under the isomorphism $H^\ast([\ast/\mathbb{G}_m]) \cong R[\tau]$. Therefore
\begin{align*}
\begin{split}
    [H_\ast(\Psi_0)(t \boxtimes u_{0,v,1})](\mu_{0,v_1,k_1} \dots \mu_{0,v_N,k_N}) & =  (t \boxtimes u_{0,v,1} )(\sum_{0 \leq j_i \leq k_i,i=1,...,N} \frac{1}{j_1!...j_N!} \\ \tau^{j_1+...+j_N}  \boxtimes (\mu_{0,v_1,k_1-j_1}...\mu_{0,v_N,k_N-j_N})) \\
    & = \begin{cases}
        1, & \text{if } N=1, v_N=v \text{ and } k_1=2 \\
        0, & {\rm otherwise}
    \end{cases}
\end{split}
\end{align*}
so that $H_\ast(\Psi_0)(t \boxtimes u_{0,v,1}) = u_{0,v,2}$. Iteration gives $H_\ast(\Psi_0)( t^k \boxtimes u_{0,v,1}) = u_{0,v,k+1}$.
\end{proof}

\begin{lem}
\label{hom6lem7}
Let $X$ be in class D. Then for all $\al \in K^0_{\rm sst}(X)$, $v \in Q$, and $\eta \in \hat{H}_\ast(\M_\al)$
	 \begin{equation*}
	 	\begin{split}
	 	Y(u_{0,v,1},z)\eta = (-1)^{{\rm deg}(v) \chi(\al,\al)} \big \{\sum_{i \geq 0} z^i \cdot u_{\al,v,i+1} \eta + \sum_{w \in Q, k \geq 0} k! \chi(v,w) z^{-k-1} (\eta \\ \cap \mu_{\al,w,k}) \big \}
		\end{split}
	 \end{equation*}
if $X$ is $2n$-Calabi--Yau and
    	\begin{equation*}
		\begin{split}
		Y(u_{0,v,1},z)\eta =(-1)^{{\rm deg}(v) \chi(\al,\al)} \big \{ \sum_{i \geq 0} z^i \cdot u_{\al,v,i+1} \eta + \sum_{w \in Q, k \geq 0} k! \chi_{\rm sym}(v,w) z^{-k-1}  \\ ( \eta  \cap \mu_{\al,w,k}) \big \}
		\end{split}
	\end{equation*}
  otherwise.  
\end{lem}
\begin{proof}
Assume that $X$ is $2n$-Calabi--Yau. When $X$ is not $2n$-Calabi--Yau, the proof is essentially identical.
By definition
    \begin{equation*}
    \begin{split}
        Y(u_{0,v,1},z)\eta  = \ep_{0,\al}(-1)^{{\rm deg}(v) \chi(\al,\al)} z^{\chi(0,\al)} \hat{H}_\ast(\Phi_{0,\al}) \circ \hat{H}_\ast(\Psi_0 \times {\rm Id}_{\M_\al}) \big \{ (\sum_{i \geq 0} t^i z^i) \\ \boxtimes (u_{0,v,1} \boxtimes \eta)
     \cap {\rm exp} (\sum_{v,w \in Q} \sum_{j,k \geq 0, j+k \geq 1} (-1)^{j-1} (j+k-1)! z^{-j-k} \chi(v,w) \mu_{0,v,j} \\ \boxtimes \mu_{\al,w,k} \big\}
    \end{split}
    \end{equation*}
which gives    
\begin{equation}
\label{hom6eqn5}
    \begin{split}
        Y(u_{0,v,1}z)\eta = (-1)^{{\rm deg}(v) \chi(\al,\al)} \hat{H}_\ast(\Phi_{0,\al}) \circ \hat{H}_\ast(\Psi_0 \times {\rm Id}_{\M_\al}) \big \{(\sum_{i \geq 0} z^it^i) \boxtimes  (u_{0,v,1} \boxtimes \\ \eta)   \cap [1 + \sum_{v,w\in Q, k \geq 0}  k! \chi(v,w) z^{-k-1}  \mu_{0,v,1} \boxtimes \mu_{\al,w,k} \big \}
    \end{split}
\end{equation}
because, by Lemma \ref{hom6lem5}, $u_{0,v,1} \cap \mu_{0,v,j} = 0$ unless $j=1$, in which case $u_{0,v,1} \cap \mu_{0,v,1} = 1$ and $u_{0,v,1} \cap \mu_{0,v,1}^k = 0$ for all $k > 1$. So, the right hand side of (\ref{hom6eqn5}) equals
\begin{equation}
\label{hom6eqn6}
\begin{split}
(-1)^{{\rm deg}(v) \chi(\al,\al)} \hat{H}_\ast(\Phi_{0,\al}) \big \{ \hat{H}_\ast(\Psi_0)((\sum_{i \geq 0} z^i t^i) \boxtimes u_{0,v,1}) \boxtimes \eta + \hat{H}_\ast(\Psi_0) (\sum_{i \geq 0} z^it^i \boxtimes \\ 1_0) \boxtimes (\sum_{w \in Q} k! \chi(v,w) z^{-k-1}  \cdot (\eta \cap \mu_{\al,w,k})) \big \}
\end{split}     
\end{equation}
By Lemma \ref{hom6lem6},
$$\hat{H}_\ast(\Psi_0)(\sum_{i \geq 0} z^it^i \boxtimes \mu_{0,v,1}) = \sum_{i \geq 0} z^i u_{0,v,i+1} \text{ and } \hat{H}_\ast(\Psi_0)(\sum_{i \geq 0} z^i t^i \boxtimes 1_0) = 1_0.$$
Substituting this into (\ref{hom6eqn6}) completes the proof.
\end{proof}

\begin{thm}
\label{hom6thm8}
Let $X$ be in class D. If $X$ is $2n$-Calabi--Yau then the isomorphism of graded Hopf algebras
    \begin{equation}
    \label{hom6eqn7}
        \hat{H}_\ast(\M,\Q) \cong V_{K_{\rm top}^0(X^{\rm an}) \oplus K_{\rm top}^1(X^{\rm an})},
    \end{equation}
is moreover an isomorphism of graded vertex algebras where the right hand side of (\ref{hom6eqn7}) denotes the graded generalized super-lattice vertex $R$-algebra on the generalized super-lattice $(K_{\rm top}^0(X^{\rm an}) \oplus K_{\rm top}^1(X^{\rm an}), \chi)$ and the map $K_{\rm sst}^0(X) \hookrightarrow K_{\rm top}^0(X^{\rm an})$. Up to isomorphism, this vertex algebra is independent of the representative of the group cohomology class $[\ep] \in H^2(K^0_{\rm sst}(X),\Z/2\Z)$ that a given choice of $\{\ep_{\al,\be}\}_{\al,\be \in K^0_{\rm sst}(X)}$ defines. If $X$ is not $2n$-Calabi--Yau then (\ref{hom6eqn7}) is an isomorphism of graded vertex $R$-algebras where the right hand side of (\ref{hom6eqn7}) is the graded generalized super-lattice vertex $R$-algebra associated to the generalized super-lattice $(K_{\rm top}^0(X^{\rm an})\oplus K^1_{\rm top}(X^{\rm an}), \chi_{\rm sym})$ and the map $K_{\rm sst}^0(X) \hookrightarrow K_{\rm top}^0(X^{\rm an})$.
\end{thm}
\begin{proof}
By Proposition 3.7 it suffices to study fields which are of the form $Y(u_{0,v,1},z)$ for $v \in Q$. We will assume that $v \in Q^+$. The $v \in Q^-$ case is nearly identical.

First, suppose $n <0$. For $\eta \in \hat{H}_\ast(\M_\al)$, by Lemma \ref{hom6lem7},
    $$Y(u_{0,v,1},z)_n \eta = (-1)^{{\rm deg}(v) \chi(\al,\al)} u_{\al,v,i} \cdot \eta,$$
which corresponds to $v_n (\eta)$.  For $n=0$,   
    \begin{align*}
        Y(u_{0,v,1},z)_0 \eta & =  (-1)^{{\rm deg}(v) \chi(\al,\al)} \eta \cap \mu_{\al,v,0} 
        = 0 
        = (-1)^{{\rm deg}(v) \chi(\al,\al)}\chi(\al,0) \cdot \eta.
    \end{align*}
For $n > 0$ and $w \in Q$
    \begin{align*}
        Y(u_{0,v,1},z)_n u_{\al,w,i} & = (-1)^{{\rm deg}(v) \chi(\al,\al)} n! \chi(v,w) (u_{\al,v,i} \cap \mu_{\al,w,n}) \\
        & = \begin{cases}
           (-1)^{{\rm deg}(v) \chi(\al,\al)}  n \chi(v,w), & i = n, \\
            0, & \text{otherwise.}
        \end{cases}
    \end{align*}
This identifies $ Y(u_{0,v,1},z)_n$ with $v_n(-)$. 
\end{proof}

\end{document}